\documentclass[11pt,twoside]{article}
\usepackage{times}
\usepackage{amsmath,amssymb,amsthm}
\usepackage{enumerate}
\usepackage{cite}

\allowdisplaybreaks

\newcommand{\R}{\mathbb R}

\newcommand{\p}{\partial}
\newcommand{\ve}{\varepsilon}
\newcommand{\f}{\frac}
\newcommand{\na}{\nabla}
\newcommand{\la}{\lambda}

\renewcommand{\b}{\beta}
\newcommand{\al}{\alpha}
\renewcommand{\t}{\tilde}

\newcommand{\vp}{\varphi}
\renewcommand{\O}{\Omega}
\renewcommand{\th}{\theta}
\newcommand{\g}{\gamma}
\newcommand{\G}{\Gamma}
\newcommand{\si}{\sigma}

\newcommand{\Dl}{\Delta}

\newcommand{\ds}{\displaystyle}

\newcommand{\RN}[1]{\textup{\uppercase\expandafter{\romannumeral#1}}}

\theoremstyle{plain}
\newtheorem{theorem}{Theorem}[section]

\newtheorem{lemma}[theorem]{Lemma}

\theoremstyle{definition}

\theoremstyle{remark}
\newtheorem{remark}{Remark}[section]

\newtheorem*{acknowledgement}{Acknowledgement}

\numberwithin{equation}{section}

\title{On the global existence and blowup of smooth solutions of
  3-D compressible Euler equations with time-depending damping}

\author{Fei Hou$^{1, *}$, \qquad Ingo Witt$^{2, *}$, \qquad Huicheng
  Yin$^{3, }$\footnote{Fei Hou (\texttt{houfeimath@gmail.com}) and
    Huicheng Yin (\texttt{huicheng$@$nju.edu.cn}) were supported by
    the NSFC (No.~11571177) and the Priority Academic Program
    Development of Jiangsu Higher Education Institutions. Ingo Witt
    (\texttt{iwitt@uni-math.gwdg.de}) was partly supported by the DFG
    via the Sino-German project ``Analysis of Partial Differential
    Equations and Applications."  This research was carried out when
    Huicheng Yin and Fei Hou were visiting the Mathematical Institute
    of the University of G\"ottingen in February 2013 and November
    2014, respectively.}\\[12pt] {\small 1. Department of Mathematics and IMS,
  Nanjing University, Nanjing 210093, China}\\ {\small 2. Mathematical
  Institute, University of G\"ottingen, Bunsenstr. 3-5, 37073
  G\"ottingen, Germany}\\ {\small 3. School of Mathematical Sciences, Nanjing
  Normal University, Nanjing 210023, China}}

\begin{document}

\maketitle
\thispagestyle{empty}

\begin{abstract}
In this paper, we are concerned with the global existence and blowup
of smooth solutions of the 3-D compressible Euler equation with
time-depending damping
\begin{equation*}
\left\{ \enspace
\begin{aligned}
  &\p_t\rho+\operatorname{div}(\rho u)=0,\\
  &\p_t(\rho u)+\operatorname{div}\left(\rho u\otimes
  u+p\,\RN{1}_{3}\right)=-\,\ds\f{\mu}{(1+t)^{\la}}\,\rho u,\\
  &\rho(0,x)=\bar \rho+\ve\rho_0(x),\quad u(0,x)=\ve u_0(x),
\end{aligned}
\right.
\end{equation*}
where $x\in\R^3$, $\mu>0$, $\la\ge 0$, and $\bar\rho>0$ are constants,
$\rho_0,\, u_0\in C_0^{\infty}(\R^3)$, $(\rho_0, u_0)\not\equiv 0$,
$\rho(0,\cdot)>0$, and $\ve>0$ is sufficiently small. For
$0\le\la\le1$, we show that there exists a global smooth solution
$(\rho, u)$ when $\operatorname{curl} u_0\equiv 0$, while for $\la>1$,
in general, the solution $(\rho, u)$ will blow up in finite
time. Therefore, $\la=1$ appears to be the critical value for the
global existence of small amplitude smooth solutions.


\end{abstract}


\section{Introduction}
In this paper, we are concerned with the global existence and blowup
of smooth solutions of the three-dimensional compressible Euler
equations with time-depending damping
\begin{equation}\label{1.1}
\left\{ \enspace
\begin{aligned}
&\p_t\rho+\operatorname{div}(\rho u)=0,\\
&\p_t(\rho u)+\operatorname{div}(\rho u\otimes
  u+p\,\RN{1}_{3})=-\,\f{\mu}{(1+t)^{\la}}\,\rho u,\\
&\rho(0,x)=\bar \rho+\ve\rho_0(x),\quad u(0,x)=\ve u_0(x),
\end{aligned}
\right.
\end{equation}
where $x=(x_1, x_2, x_3)\in\R^3$, $\rho$, $u = (u_1, u_2, u_3)$, and
$p$ stand for the density, velocity, and pressure, respectively,
$\RN{1}_{3}$ is the $3\times 3$ identity matrix, and $u_0=(u_{1,0},
u_{2,0}, u_{3,0})$. The equation of state of the gas is assumed to be
$p(\rho)=A\rho^{\g}$, where $A>0$ and $\g>1$ are
constants. Furthermore, $\mu>0$, $\la\ge 0$, and $\bar\rho>0$ are
constants, $\rho_0,\, u_0\in C_0^{\infty}(\R^3)$, $(\rho_0,
u_0)\not\equiv 0$, $\rho(0, \cdot)>0$, and $\ve>0$ is sufficiently
small.

For $\mu=0$, \eqref{1.1} is the standard compressible Euler
equation. It is well known that smooth solutions $(\rho, u)$ of
\eqref{1.1} will in general blow up in finite time. For the extensive
literature on blowup results and the blowup mechanism for $(\rho, u)$,
see \cite{1,2,3,4,5,6,17,18,20,25,26} and the references therein.

For $\la=0$, it has been shown that \eqref{1.1} admits a global smooth
solution and the long-term behavior of the solution $(\rho,u)$
has been established, see \cite{11,16,19,21,22}.

For $\mu>0$ and $\la>0$, an interesting problem arises: does the
smooth solution of \eqref{1.1} blow up in finite time or does it exist
globally? In this paper, we will systematically study this problem
under the assumption that $\operatorname{curl}u_0=
\left(\p_2u_{3,0}-\p_3u_{2,0}, \p_3u_{1,0}-\p_1u_{3,0},
\p_1u_{2,0}-\p_2u_{1,0}\right)\equiv0$. In this case it is not hard to
see that $\operatorname{curl} u(t,\cdot)\equiv 0$ for all $t\ge 0$ as
long as the smooth solution $(\rho, u)$ of \eqref{1.1} exists. Then
one can introduce a potential function $\vp=\vp(t,x)$ such that
$u=\na\vp$ (here and below, $\na=\na_x$), where the $C^\infty$ scalar
function $\vp$ has compact support in $x$ (as $u(t,\cdot)$ has compact
support in view of $u_0\in C_0^{\infty}(\R^3)$ and finite propagation
speed which holds for hyperbolic systems). Substituting $u=\na\vp$
into the second equation of \eqref{1.1}, we obtain
\begin{equation}\label{1.2}
  \p_t\vp+\f12|\na\vp|^2+h(\rho)+\f{\mu}{(1+t)^{\la}}\,\vp=0,
\end{equation}
where $\ds h'(\rho)=c^2(\rho)/\rho$ with $c(\rho)=\sqrt{p'(\rho)}$ and
$h(\bar\rho)=0$. From $h'(\rho)>0$ for $\rho>0$
we have that
\begin{equation}\label{1.3}
  \rho=h^{-1}\left(-\left(\p_t\vp+\f12\,|\na\vp|^2
  +\f{\mu}{(1+t)^{\la}}\,\vp\right)\right),
\end{equation}
where $\bar\rho=h^{-1}(0)$ and $h^{-1}$ is the inverse function of
$h=h(\rho)$. Substituting \eqref{1.3} into the first equation of
\eqref{1.1} yields
\begin{multline}\label{1.4}
  \p_t^2\vp-c^2(\rho)\Delta\vp+2\sum_{k=1}^3\left(\p_k\vp\right)
  \p_{tk}\vp
  +\sum_{i,k=1}^3\left(\p_i\vp\right)\left(\p_k\vp\right)\p_{ik}\vp
  \\ +\f{\mu}{(1+t)^{\la}}\,|\na\vp|^2
  +\p_t\left(\f{\mu}{(1+t)^{\la}}\,\vp\right)=0.
\end{multline}

As for the initial data $\vp(0,\cdot)$ and $\p_t\vp(0,\cdot)$ for
Eq.~\eqref{1.4}: Obviously, $\vp(0,\cdot)=\ve\vp_0$, where $\ds
\vp_0(x)=\int_{-\infty}^{x_1}u_{1,0}(s,x_2,x_3)\,ds$. Note that
$\vp_0\in C_0^{\infty}(\R^3)$ in view of $\operatorname{curl}u_0\equiv
0$ and $u_0\in C_0^{\infty}(\R^3)$. Furthermore, from Eq.~\eqref{1.2} we
infer that $\p_t\vp(0,\cdot)=\ve\vp_1+\ve^2g(\cdot,\ve)$, where $\ds
\vp_1=-\left(\mu \vp_0 + \f{c^2(\bar\rho)}{\bar\rho}\,\rho_0\right)$
and
\[
  g(x,\ve)=-\rho^2_0(x)\int_0^1\left(\f{c^2(\rho)}{\rho}\right)'\biggr|_{\rho=
    \bar\rho+\th\ve\rho_0(x)}\,d\th - \f12\,\sum_{i=1}^3u^2_{i,0}(x).
\]
Notice that $g(x,\ve)$ is smooth in $(x, \ve)$ and has compact support
in $x$. Consequently, studying problem~\eqref{1.1} under the
assumption $\operatorname{curl}u_0\equiv 0$ is equivalent to
investigating the problem
\begin{equation}\label{1.5}
\left\{ \enspace
\begin{aligned}
  &\p_t^2\vp-c^2(\rho)\,\Delta\vp+2\sum_{k=1}^3(\p_k\vp)\p_{tk}\vp
  +\sum_{i,k=1}^3(\p_i\vp)(\p_k\vp)\p_{ik}\vp \\ & \hspace*{70mm}
  +\f{\mu}{(1+t)^{\la}}\,|\na\vp|^2+\p_t\left(\f{\mu}{(1+t)^{\la}}\vp\right)=0,\\
  &\vp(0,x)=\ve\vp_0(x),\quad \p_t\vp(0,x)=\ve\vp_1(x)+\ve^2g(x,\ve).
\end{aligned}
\right.
\end{equation}
Here we mention that
\[
  c^2(\rho) = c^2(\bar \rho) -\left(\gamma-1\right) \left(\p_t\vp
  +\frac12\,|\nabla \vp|^2 + \frac\mu{(1+t)^\lambda}\,\vp\right)
\]
which follows by direct computation.

We now state the first main result of this paper.

\begin{theorem}[Global existence for $0\le\la\le 1$]\label{thm1.1}
Suppose that $\operatorname{curl}u_0\equiv0$. If $\mu>0$ and $0\le
\la\le 1$, then, for $\ve>0$ small enough, \eqref{1.5} admits a global
smooth solution $\vp$. As a consequence, \eqref{1.1} has a global
smooth solution $(\rho, u)$ which fulfills $\rho>0$ and which is
uniformly bounded for $t\ge 0$ together with all its derivatives.
\end{theorem}

\begin{remark}\label{rem1.1}
The principal part of the linearization of the equation in \eqref{1.5}
about $(\rho,\vp)=(\bar\rho,0)$ is
\begin{equation}\label{1.6}
  {\mathcal L}(\dot\vp)\equiv\p_t^2\dot\vp-c^2(\bar\rho)\,
  \Delta\dot\vp+\f{\mu}{(1+t)^{\la}}\,
  \p_t\dot\vp-\f{\mu\la}{(1+t)^{\la+1}}\,\dot\vp.
\end{equation}
For the linear operator ${\mathcal L_0}$ with
\begin{equation*}
  {\mathcal L_0}(\dot\vp)\equiv
  \p_t^2\dot\vp-c^2(\bar\rho)\Delta\dot\vp+\ds\f{\mu}{(1+t)^{\la}}\p_t\dot\vp,
\end{equation*}
which appears as part of \eqref{1.6}, it is shown in \cite{23,24} that
the large-term behavior of solutions $\vp$ of ${\mathcal
  L_0}(\dot\vp)=0$ depends on the value of $\la$. For $0\le\la<1$ is
the same as the large-term behavior of solutions of the linear heat
equation $\p_t\dot\vp-c^2(\bar\rho)\Delta\dot\vp=0$, while for $\la>1$
it is the same as the large-term behavior of solutions of the linear
wave equation $\p_t^2\dot\vp-c^2(\bar\rho) \Delta\dot\vp=0$. Moreover,
precise microlocal large-term decay properties of solutions $\dot\vp$
of ${\mathcal L}(\dot\vp)=0$ have been established in \cite{15} for a
special range of values of $\la$ and $\mu$. It seems to be difficult,
however, to apply these microlocal estimates to attack the quasilinear
problem \eqref{1.5}. (In general, those microlocal estimates are
useful when treating semilinear damped wave equations, see
\cite{8,9} and references therein.)
\end{remark}

\begin{remark}\label{rem1.2}
For the 1-D Burgers equation with damping term
\begin{equation}\label{1.8}
\left\{ \enspace
\begin{aligned}
&\p_t w+w\p_x w=-\,\ds\f{\mu}{(1+t)^{\la}}\,w,\qquad
  (t,x)\in\R_+\times\R,\\ &w(0,x)=\ve w_0(x),
\end{aligned}
\right.
\end{equation}
where $\mu>0$ and $\la\ge 0$ are constants, $w_0\in C_0^{\infty}(\R)$,
$w_0\not\equiv 0$, and $\ve>0$ is sufficiently small, one concludes by
the method of characteristics that
\begin{equation*}
\left\{ \enspace
\begin{aligned}
&T_{\ve}=\infty & \text{if $0\le\la<1$ or $\la=1$,
  $\mu>1$,}\\ &T_{\ve}<\infty & \text{if $\la>1$ or $\la=1$,
  $0\le\mu\le 1$,}
\end{aligned}
\right.
\end{equation*}
where $T_{\ve}$ is the lifespan of the smooth solution $w=w(t,x)$ of
\eqref{1.8}. Therefore, $\la=1$ again appears to be the critical value
for the global existence of smooth solutions $w$ of \eqref{1.8} in the
presence of the damping term $\ds\f{\mu}{(1+t)^{\la}}\,w$.
\end{remark}

\begin{remark}\label{rem1.3}
The smallness of $\ve>0$ in Theorem~\ref{thm1.1} is necessary in order
to guarantee the global existence of smooth solution $(\rho,
u)$. Indeed, as in \cite{19}, large amplitude smooth solution of
\eqref{1.1} may blow up in finite time even for $0\le\la\le 1$. See
also Theorem~\ref{thm4.1} in Chapter~\ref{chap4}.
\end{remark}

Next we concentrate on the case of $\la>1$. As in \cite{18}, introduce
the two functions
\begin{align*}
  q_0(l)&=\int_{|x|>l}\frac{(|x|-l)^2}{|x|}\left(\rho(0,x)-\bar\rho\right)dx, \\
q_1(l)&=\int_{|x|>l}\frac{|x|^2-l^2}{|x|^3}\,x\cdot(\rho u)(0,x)\,dx. 
\end{align*}

\begin{theorem}[Blowup for $\la>1$]\label{thm1.2}
Suppose $\operatorname{supp}\rho_0,\,\operatorname{supp}u_0\subseteq
\{x\colon|x|\leq M\}$ and let
\begin{align}
q_0(l)&>0, \label{1.11}\\
q_1(l)&\geq0 \label{1.12}
\end{align}
hold for all $l\in (M_0, M)$, where $M_0$ is some fixed constant
satisfying $0\leq M_0<M$. If $\mu>0$ and $\la>1$, then there exists an
$\ve_0>0$ such that, for $0<\ve\leq\ve_0$, the lifespan $T_\ve$ of the
smooth solution $(\rho, u)$ of \eqref{1.1} is finite.
\end{theorem}

\begin{remark}\label{rem1.4}
It is not hard to find a large number of initial data $(\rho,
u)(0,\cdot)$ such that both \eqref{1.11} and \eqref{1.12} are
satisfied. In addition, we would like to underline that
Theorem~\ref{thm1.2} holds also for
$\operatorname{curl}u_0\not\equiv0$.
\end{remark}

\bigskip

Let us indicate the proofs of Theorems~\ref{thm1.1}
and~\ref{thm1.2}. To prove Theorem~\ref{thm1.1}, we first introduce
the function $\psi=\ds\f{\vp}{(1+t)^{\la}}$ which fulfills the
second-order quasilinear wave equation
\[
\p_t^2\psi-\Dl\psi+\ds\f{\mu}{(1+t)^\la}\p_t\psi+\f{2\la}{1+t}\p_t\psi
-\f{\la(1-\la)}{(1+t)^2}\psi=Q(\psi),
\]
where $Q(\psi)$ stands for an error term which is of the second order
in $(\psi, \p\psi, \p^2\psi)$; $\p=(\p_t,\nabla)$. Then, in order to
establish the global existence of $\psi$, we introduce the
time-weighted energy
\[
E_N(\psi)(t)=\sum_{0\leq |a|\leq N}\int_{\mathbb
  R^3}\left((1+t)^{2\la}|\p \G^a\psi|^2+|\G^a\psi|^2 \right)dx,
\]
where $N\ge 8$ is a fixed number, $\G=(\G_0, \G_1, \dots, \G_7)=(\p,
\O, S)$ with $\O=(\O_1, \O_2, \O_3)=x\wedge \na$,
$S=t\p_t+\ds\sum_{k=1}^3x_k\p_k$, and
$\G^a=\G_0^{a_0}\G_1^{a_1}\dots\G_7^{a_7}$.  Note that the vector
fields $\G$ which appear in the definition of the energy
$E_N(\psi)(t)$ only comprise part of the standard Klainerman vector
fields $\{\p, \O, S, H\}$, where $H=(H_1, H_2, H_3)=(x_1\p_t+t\p_1,
x_2\p_t+t\p_2, x_3\p_t+t\p_3)$. This is due to the fact that the
equation in \eqref{1.5} is not invariant under the Lorentz
transformations $H$ in view of the presence of the time-depending
damping.  By a rather technical and involved analysis of the resulting
equation for $\psi$, we eventually show that $\ds E_N(\psi)(t)\le
\f12\, K^2\ve^2$ when $E_N(\psi)(t)\le K^2\ve^2$ is assumed for some
suitably large constant $K>0$ and small $\ve>0$. Based on this and a
continuous induction argument, the global existence of $\psi$ and then
Theorem~\ref{thm1.1} are established for $0\le\la\le 1$. To prove the
blowup result of Theorem~\ref{thm1.2} for $\la>1$, as in \cite{18}, we
derive a related second-order ordinary differential inequality. From
this and assumptions \eqref{1.11}-\eqref{1.12}, an upper bound of the
lifespan $T_{\ve}$ is derived by making essential use of $\la>1$. In
this way the proof of Theorem~\ref{thm1.2} is completed.  Finally, in
Theorem~\ref{thm4.1}, we show that for large data smooth solution
$(\rho, u)$ of \eqref{1.1}, even in case $0\le\la\le 1$, $(\rho, u)$
will in general blow up in finite time.

\smallskip


\bigskip
Throughout, we shall use the following notation and conventions:

\begin{itemize}
\item $\na$ stands for $\na_x$,

\item $r=|x|=\sqrt{x_1^2+x_2^2+x_3^2}$,

\item $\langle r-t\rangle=\bigl(1+(r-t)^2\bigr){}^{1/2}$,

\item $\ds \|u(t,\cdot)\|=\biggl(\int_{\R^3}|u(t,x)|^2dx\biggr)^{1/2}$
  and $\|u(t,\cdot)\|_{L^{\infty}}=\ds\sup_{x\in\R^3}|u(t,x)|$,

\item $Z$ denotes one of the Klainerman vector fields $\{\p, S, \O,
  H\}$ on $\R_+\times\R^3$, where $\p=(\p_t, \na)$, $\ds
  S=t\p_t+\sum_{k=1}^3x_k\p_k$, $\O=(\O_1, \O_2, \O_3)=x\wedge \na$,
  and $H=(H_1, H_2, H_3)=(x_1\p_t+t\p_1, x_2\p_t+t\p_2,
  x_3\p_t+t\p_3)$,

\item $\G$ denotes one of the vector fields $\{\p, S, \O\}$ on
  $\R_+\times\R^3$.

\item $\beta$ is the solution of $\ds \beta'(t) = \frac{\mu}{(1+t)^\la}\,
  \beta(t)$ for $t\geq0$, $\beta(0)=1$, i.e.,
\begin{equation}\label{beta}
\b(t)\equiv\begin{cases} e^{\f{\mu}{1-\la}[(1+t)^{1-\la}-1]}, &
\la\geq 0,\,\la\neq1,\\ (1+t)^\mu, & \la=1.
\end{cases}
\end{equation}

\item 
  $c(\bar\rho)=1$ will be assumed throughout (introduce
  $X=x/c(\bar\rho)$ as new space coordinate if necessary).
\end{itemize}


\section{Global existence for small data in case \boldmath $0\le\la\le1$}
\label{chap2}

Throughout this section, $C>0$ stands for a generic constant which is
independent of $K$, $\ve$, and $t$.

\smallskip

We start by recalling the following Sobolev-type inequality (see \cite{12}):

\begin{lemma}\label{lem2.1}
Let $u=u(t,x)$ be a smooth function of $(t,x)\in
[0,\infty)\times\R^3$. Then
\begin{equation}\label{2.1}
|u(t,x)| \leq C(1+r)^{-1}\sum_{|a|\leq 2}\|\G^a u(t,x)\|.
\end{equation}
\end{lemma}

Moreover, we shall make use of the following inequalities (see
\cite[Lemmas~2.3 and~3.1, Theorem~5.1]{13}):

\begin{lemma}\label{lem2.2}
For $u\in C^2([0,\infty)\times\R^3)$,
\begin{align}
  \langle r-t\rangle\|\na\p u(t,x)\| &\leq
  C\,\bigg(\sum_{|b|\leq 1}\|\p\G^bu(t,x)\|+t\|\square u(t,x)\|\biggr),
  \label{2.4} \\ (1+r)\langle r-t\rangle|\na\p u(t,x)| &\leq
  C\,\bigg(\sum_{|b|\leq 3}\|\p\G^bu(t,x)\|+t\|\square
  u(t,x)\|\biggr),\label{2.5}
\end{align}
where $\square=\p_t^2-\Delta=\p_t^2-\ds\sum_{k=1}^3\p_k^2$.
\end{lemma}

We now reformulate problem \eqref{1.5}. Let $\ds \psi=
\f{\vp}{(1+t)^\la}$. From \eqref{1.5} and $c(\bar\rho)=1$ we then have
\begin{equation}\label{2.6}
  \square\psi+\f{\mu}{(1+t)^\la}\p_t\psi+\f{2\la}{1+t}\p_t\psi
  -\f{\la(1-\la)}{(1+t)^2}\psi=Q(\psi),
\end{equation}
where
\begin{multline*}\label{2.7}
Q(\psi)=
(c^2(\rho)-1)\Dl\psi-2(1+t)^\la\p_t\na\psi\cdot\na\psi-2\la(1+t)^{\la-1}|\na\psi|^2
\\ -\mu|\na\psi|^2-(1+t)^{2\la}\ds\sum_{1\le i,j\le
  3}(\p_i\psi)(\p_j\psi)\p_{ij}\psi.
\end{multline*}
We define a time-weighted energy for Eq.~\eqref{2.6},
\begin{eqnarray*}
E_N(\psi(t)) = \ds\sum_{0\leq |a| \leq N}\int_{\R^3}
\bigg((1+t)^{2\la}|\p\G^a\psi|^2+|\G^a\psi|^2\biggr)dx,
\end{eqnarray*}
where $N\ge 8$ is a fixed number. Moreover, we assume that for any $t\ge 0$
\begin{equation}\label{2.8}
  E_N(\psi(t))\leq K^2\ve^2,
\end{equation}
where $K>0$ is a suitably large constant. It follows from \eqref{2.1}
and \eqref{2.8} that, for all $|a|\leq N-2$,
\begin{equation}\label{2.9}
\begin{aligned}
  |\p\G^a\psi| & \leq C(1+r)^{-1}\sum_{|b|\leq
    2}\|\G^b\p\G^a\psi(t,x)\| \leq C(1+r)^{-1}\sum_{|b|\leq
    N}\|\p\G^b\psi(t,x)\| \\ & \leq
  C(1+r)^{-1}(1+t)^{-\la}\sqrt{E_N(\psi(t))} \leq
  CK\ve(1+r)^{-1}(1+t)^{-\la}
\end{aligned}
\end{equation}
and
\begin{equation}\label{2.10}
|\G^a\psi|\leq C(1+r)^{-1}\sum_{|b|\leq N}\|\G^b\psi(t,x)\|\leq  CK\ve(1+r)^{-1}.
\end{equation}

In view of Lemma~\ref{lem2.2} and \eqref{2.8}, we have

\begin{lemma}\label{lem2.3}
Let $\psi$ be a solution of \eqref{2.6}, Then, for all $|a|\leq N-3$
and $0\le\la\le 1$, we have the pointwise estimate
\begin{equation}\label{2.11}
\|\na\p\G^a\psi\|_{L^{\infty}} \leq CK\ve(1+t)^{-2\la}.
\end{equation}
Moreover, for $0\le l\leq N-1$, the weighted $L^2$ estimate
\begin{multline}\label{2.12}
\sum_{|b|\le l}\|\langle r-t\rangle\na\p\G^b\psi(t,x)\| \\ \leq
C\sum_{|c|\leq l+1}\|\p\G^c\psi(t,x)\| +C(1+t)^{1-\la}\ds\sum_{|c|\leq
  l}\|\na\G^c\psi(t,x)\| +C(1+t)^{-1}\sum_{|c|\leq
  l}\|\G^c\psi\|
\end{multline}
holds.
\end{lemma}
\begin{proof}
It follows from \eqref{2.5}--\eqref{2.6} and \eqref{2.9}--\eqref{2.10} that
\begin{multline*}
(1+t)\sum_{|a|\leq N-3}|\na\p\G^a\psi|\\
\begin{aligned}
&\le C\sum_{|a|\leq N-3}(1+r)\langle r-t\rangle|\na\p\G^a\psi| \\
&\leq C\sum_{|c|\leq N}\|\p\G^c\psi\|+Ct\sum_{|a|\leq N-3}\|\square\G^a\psi\|\\
&\leq CK\ve(1+t)^{-\la}+C(1+t)^{1-\la}\sum_{|a|\leq N-3}\|\p_t\G^a\psi\|+C(1+t)^{-1}\sum_{|a|\leq N-3}\|\G^a\psi\| \\
&\qquad +C(1+t)\sum_{|b|+|c|\leq N-3}\|\na\p\G^b\psi\G^c\psi\|+C(1+t)^{1+\la}\sum_{|a|\leq N-3}\|\G^a(\p_t\na\psi\cdot\na\psi)\| \\
&\leq CK\ve(1+t)^{1-2\la}+CK\ve(1+t)\sum_{|a|\leq N-3}\|\na\p\G^a\psi\|_{L^{\infty}},
\end{aligned}
\end{multline*}
which derives \eqref{2.10} in view of the smallness of $\ve>0$.

By \eqref{2.4}, \eqref{2.9}-\eqref{2.11} and Eq.~\eqref{2.6}, we
have that, for $l\le N-1$,
\begin{multline}\label{2.13}
\sum_{|b|\le l}\|\langle r-t\rangle\na\p\G^b\psi\| \\
\begin{aligned}
&\leq C\sum_{|c|\leq l+1}\|\p\G^c\psi\|+Ct\sum_{|b|\leq l}\|\G^b\square\psi\|\\
&\leq C\sum_{|c|\leq l+1}\|\p\G^c\psi\|+C(1+t)^{1-\la}\sum_{|c|\leq l}\|\na\G^c\psi\|+C(1+t)^{-1}\sum_{|c|\leq l}\|\G^c\psi\| \\
  &\qquad +C(1+t)^{1+\la}\sum_{|b|\leq l}\|\G^b(\p_t\na\psi\cdot\na\psi)\|\\
&\qquad+C(1+t)\sum_{\substack{|c|\leq N-3, \\ |b|\leq l-|c|}} \|\langle r-t\rangle^{-1}\G^c\psi\|_{L^\infty}\|\langle r-t\rangle\na\p\G^b\psi\| \\
&\qquad +C(1+t)\sum_{\substack{2-N\leq|c|\leq l, \\ |b|\leq l+2-N}}\|(1+r)\na\p\G^b\psi\|_{L^\infty}\|(1+r)^{-1}\G^c\psi\|\\
&\le C\sum_{|c|\leq l+1}\|\p\G^c\psi\|+C(1+t)^{1-\la}\sum_{|c|\leq l}\|\na\G^c\psi\| +C(1+t)^{-1}\sum_{|c|\leq l}\|\G^c\psi\|\\
&\qquad +CK\ve \ds\sum_{|b|\le l}\|\langle r-t\rangle\na\p\G^b\psi\| +CK\ve(1+t)^{1-\la}\sum_{2-N\leq|c|\leq l}\|(1+r)^{-1}\G^c\psi\|.
\end{aligned}
\end{multline}
Note that $\G^c\psi(t,x)$ is supported in $\{x\colon|x|\leq t+M\}$. Then it
follows from Hardy inequality that
\begin{equation}\label{2.14}
  \|(1+r)^{-1}\G^c\psi\|\leq C\|\na\G^c\psi\|.
\end{equation}
Substituting \eqref{2.14} into \eqref{2.13} and applying the smallness
of $\ve$, we derive \eqref{2.12}.
\end{proof}

Next we derive the time-weighted energy estimate for the solution
$\psi$ of \eqref{2.6}.

\begin{lemma}\label{lem2.4}
Let $\mu>0$ and $\la\in(0,1]$. Under\/ \textup{assumption
  \eqref{2.8}}, for all $t>0$ and $N\ge 8$, it holds that
\begin{multline}\label{2.15}
\sum_{0\leq|a|\leq N}\int_{\R^3}\biggl((1+t)^{2\la}|\p\p^a\psi|^2+\psi^2\biggr)dx
+C\sum_{0\leq|a|\leq N}\int_0^t\int_{\R^3}(1+\tau)^\la|\p\p^a\psi|^2dxd\tau \\
\leq  C\ve^2+C(1+K\ve)\int_0^tA(\tau)\sum_{0\leq|a|\leq N}\int_{\R^3}\biggl((1+\tau)^{2\la}|\p\p^a\psi|^2+\psi^2\biggr)dxd\tau,
\end{multline}
where $A(\cdot)$ stands for a generic non-negative function such that
$A\in L^1((0,\infty))$ and $\|A\|_{L^1}$ is independent of $K$.
\end{lemma}

\begin{proof}
First we show \eqref{2.15} in case $|a|=0$. Multiplying
Eq.~\eqref{2.6} by $m(1+t)^{2\la}\p_t\psi +(1+t)^{2\la-1}\psi$ yields
by a direct computation
\begin{multline}\label{2.16}
\f{1}{2}\p_t\biggl(m(1+t)^{2\la}|\p\psi|^2+2(1+t)^{2\la-1}\psi\p_t\psi+(\mu(1+t)^{\la-1}+2\la(1+t)^{2\la-2})\psi^2\biggr)\\
+\operatorname{div}\biggl(\cdots\biggr)+\biggl(\mu m(1+t)^\la+(\la m-1)(1+t)^{2\la-1}\biggr)(\p_t\psi)^2+(1-\la m)(1+t)^{2\la-1}|\na\psi|^2 \\
 +\mu(1-\la)(1+t)^{\la-2}\psi^2+C_1(\la-1)(1+t)^{2\la-2}\psi\p_t\psi+C_2(\la-1)(1+t)^{2\la-3}\psi^2 \\
= \bigl(m(1+t)^{2\la}\p_t\psi+(1+t)^\la\psi\bigr)Q(\psi),
\end{multline}
where the constant $m>0$ will be determined later and $C_i$ ($i=1,2)$
are suitable constants. Note that
\begin{equation}\label{2.17}
2(1+t)^{2\la-1}\psi\p_t\psi \leq \si m(1+t)^{2\la}(\p_t\psi)^2+\f{1}{\si m}(1+t)^{2\la-2}\psi^2
\end{equation}
for $\si\in(0,1)$.

To guarantee the positivity of the term $m(1+t)^{2\la}|\p\psi|^2
+2(1+t)^{2\la-1}\psi\p_t\psi+(\mu(1+t)^{\la-1}+2\la(1+t)^{2\la-2})\psi^2$
in $\p_t\bigl( \cdot \bigr)$ and of the coefficients $\mu
m(1+t)^\la+(\la m-1)(1+t)^{2\la-1}$ and $(1-\la m)(1+t)^{2\la-1}$ of
$(\p_t\psi)^2$ and $|\na\psi|^2$ in the left-hand side of
\eqref{2.16}, utilizing \eqref{2.17} with $\si=2/3$, we may choose
$m>0$ to fulfill
\[
\la<\ds\f{1}{m}<{\rm min}\{\mu+\la, 2\la\}.
\]
Then integrating \eqref{2.16} over $\R^3$ yields
\begin{multline}\label{2.18}
\f{d}{dt}\int_{\R^3}\biggl((1+t)^{2\la}|\p\psi|^2+(1+t)^{\la-1}\psi^2\biggr)dx \\
 +C\int_{\R^3}\biggl((1+t)^\la(\p_t\psi)^2+(1+t)^{2\la-1}|\na\psi|^2+(1+t)^{\la-2}\psi^2\biggr)dx \\
 \leq  A(t)\int_{\R^3}(1+t)^{\la-1}\psi^2dx+C\left|\int_{\R^3}\bigl(m(1+t)^{2\la}\p_t\psi
+(1+t)^{2\la-1}\psi\bigr)Q(\psi)dx\right|.
\end{multline}
Next we improve the time-weighted estimate of $\psi$ in the left-hand
side of \eqref{2.18}. Multiplying both sides of \eqref{2.6} by
$(1+t)^{\la}\psi$ yields by direct computation
\begin{multline*}
\p_t\biggl((1+t)^\la\psi\p_t\psi+\f{\mu}{2}\psi^2\biggr)+\operatorname{div}\biggl(\cdots\biggr)
-(1+t)^\la(\p_t\psi)^2-\la(1+t)^{\la-1}\psi\p_t\psi\\
 + (1+t)^\la|\na\psi|^2+2\la(1+t)^{\la-1}\psi\p_t\psi+\la(\la-1)(1+t)^{\la-2}\psi^2
=(1+t)^{\la}\psi Q(\psi).
\end{multline*}
From this and \eqref{2.18},  we can choose the multiplier
$m(1+t)^{2\la}\p_t\psi+(1+t)^{2\la-1}\psi+\kappa(1+t)^{\la}\psi$ for
Eq.~\eqref{2.6} with a small $\kappa>0$ and then obtain by an integration
with respect to the time variable $t$
\begin{multline}\label{2.19}
\int_{\R^3}\biggl((1+t)^{2\la}|\p\psi|^2+\psi^2\biggr)dx+C\int_0^t\int_{\R^3}(1+\tau)^\la|\p\psi|^2dxd\tau\\
\leq C\ve^2+\int_0^tA(\tau)\int_{\R^3}\psi^2dxd\tau+C\int_0^t\biggl|
\int_{\R^3}(1+\tau)^{2\la}(\p_t\psi) Q(\psi)dx\biggr|\,d\tau \\
+C\int_0^t\biggl|\int_{\R^3}(1+\tau)^\la\psi Q(\psi)dx\biggr|\,d\tau.
\end{multline}

Next we derive the time-weighted estimates of $\p^{\al}\psi$ with
$1\leq|\al|\leq N$. Taking $\p^\al$ on both sides of Eq.~\eqref{2.6}
yields
\begin{multline*}
\square\p^a\psi+\f{\mu}{(1+t)^\la}\p_t\p^a\psi+\f{2\la}{1+t}\,\p_t\p^a\psi
\\ =\p^a
Q(\psi)+\sum_{1\le|b|\leq|a|}\f{1}{(1+t)^\la}\Bigl(1+O\bigl((1+t)^{\la-1}\bigr)
\Bigr)\p^b\psi-\la(\la-1)\p^a\Bigl(\f{1}{(1+t)^2}\Bigr)\psi.
\end{multline*}
Exactly as for \eqref{2.19}, we obtain
\begin{multline}\label{2.20}
\sum_{0\le |a|\le N}\int_{\R^3}\biggl((1+t)^{2\la}|\p^{a+1}\psi|^2+\psi^2\biggr)dx
+C\sum_{0\le |a|\le N}\int_0^t\int_{\R^3}(1+\tau)^\la|\p^{a+1}\psi|^2dxd\tau\\
\leq  C\ve^2+\int_0^tA(\tau)\int_{\R^3}\psi^2dxd\tau
+C\sum_{0\le |a|\le N}\int_0^t\biggl|\int_{\R^3}(1+\tau)^{2\la}(\p_t\p^a\psi)\p^a Q(\psi)dx\biggr|\,d\tau\\
 +C\sum_{0\le |a|\le N}\int_0^t\biggl|\int_{\R^3}(1+\tau)^\la(\p^a\psi)\p^a Q(\psi)dx\biggr|\,d\tau.
\end{multline}
We now deal with the last two terms in the right-hand side of
\eqref{2.20}.  We first analyze the integrand
$(1+t)^{2\la}(\p_t\p^a\psi)\p^a Q(\psi)$ of
$\ds\int_0^t\biggl|\int_{\R^3}(1+\tau)^{2\la}(\p_t\p^a\psi)\p^a
Q(\psi)dx\biggr|\,d\tau$. Direct computation yields
\[
\p^a Q(\psi) = (c^2(\rho)-1)\Dl\p^a\psi-2(1+t)^\la\na\p_t\p^a\psi\cdot\na\psi
-(1+t)^{2\la}(\p_i\psi)(\p_j\psi)\p_{ij}\p^a\psi+\textup{l.o.t.}
\]
and
\begin{multline}\label{2.21}
 (1+t)^{2\la}(\p_t\p^a\psi)\p^a Q(\psi)
=\operatorname{div}\biggl((1+t)^{2\la}(c^2(\rho)-1)(\p_t\p^a\psi)\na\p^a\psi\biggr)-\operatorname{div}\biggl((1+t)^{3\la}|\p_t\p^a\psi|^2\na\psi\biggr)\\
 -\f{1}{2}\p_t\biggl((1+t)^{2\la}(c^2(\rho)-1)|\na\p^a\psi|^2\biggr)+(1+t)^{3\la}|\p_t\p^a\psi|^2\Dl\psi
+\la(1+t)^{2\la-1}(c^2(\rho)-1)|\na\p^a\psi|^2\\
 +\f{1}{2}(1+t)^{2\la}(c^2(\rho))'\p_t\rho|\na\p^a\psi|^2
-(1+t)^{4\la}(\p_i\psi)(\p_j\psi)(\p_{ij}\p^a\psi)\p_t\p^a\psi+\textup{l.o.t.},
\end{multline}
where here and below l.o.t. designates lower-order terms which are of
the form $(\p^{b_1}\psi)(\p^{b_2}\psi)\dots(\p^{b_l}\psi)$ (multiplied
by $\p\p^a\psi$ or $\p^a\psi$) with $l\ge 2$ and
$1\le|b_1|+\dots+|b_l|\le|a|+1$. Here we are concerned with the
top-order derivatives only.  Note that the term
$(1+t)^{4\la}(\p_i\psi)(\p_j\psi)(\p_{ij}\p^a\psi)\p_t\p^a\psi$ in
\eqref{2.21} can be expressed as
\begin{multline}\label{2.22}
  (1+t)^{4\la}(\p_i\psi)(\p_j\psi)(\p_{ij}\p^a\psi)\p_t\p^a\psi \\
=\f12\bigg\{\p_i\biggl((1+t)^{4\la}(\p_i\psi)(\p_j\psi)(\p_j\p^a\psi)\p_t\p^a\psi\biggr)
+\p_j\biggl((1+t)^{4\la}(\p_i\psi)(\p_j\psi)(\p_i\p^a\psi)\p_t\p^a\psi\biggr)\\
-\p_t\biggl((1+t)^{4\la}(\p_i\psi)(\p_j\psi)(\p_i\p^a\psi)\p_j\p^a\psi\biggr)
+\p_t\bigl((1+t)^{4\la}(\p_i\psi)\p_j\psi\bigr)(\p_i\p^a\psi)\p_j\p^a\psi+\textup{l.o.t.}
\bigg\}.
\end{multline}
Similarly, for the integrand $(1+t)^{\la}(\p^a\psi)\p^a Q(\psi)$ of
$\ds\int_0^t\biggl|\int_{\R^3}(1+\tau)^{\la}(\p^a\psi)\p^a
Q(\psi)dx\biggr|\,d\tau$, one has
\begin{multline}\label{2.23}
(1+t)^\la\p^a\psi\p^a Q(\psi) \\
\begin{aligned}
  &=\operatorname{div}\biggl((1+t)^\la(c^2(\rho)-1)\na(\p^a\psi)\p^a\psi\biggr)
-\f{1}{2}\p_i\biggl((1+t)^{3\la}(\p_i\psi)\p^a(|\na\psi|^2)\p^a\psi\biggr)\\
&\qquad -\p_t\biggl((1+t)^\la\p^a(|\na\psi|^2)\p^a\psi\biggr)
-(1+t)^\la(c^2(\rho)-1)|\na\p^a\psi|^2\\
&\qquad -(1+t)^\la (c^2(\rho))'\na\rho\cdot\na(\p^a\psi)\p^a\psi
+\la(1+t)^{\la-1}\p^a(|\na\psi|^2)\p^a\psi \\
&\qquad +(1+t)^\la\p^a(|\na\psi|^2)\p_t\p^a\psi
+\f{1}{2}(1+t)^{3\la}(\Dl\psi)\p^a(|\na\psi|^2)\p^a\psi \\
&\qquad +\f{1}{2}(1+t)^{3\la}\na\psi\cdot\na(\p^a\psi)\p^a(|\na\psi|^2)+\textup{l.o.t.}
\end{aligned}
\end{multline}
From the expression $(\p^{b_1}\psi)(\p^{b_2}\psi)\dots(\p^{b_l}\psi)$
($l\ge 2$, $1\le |b_1|+\dots +|b_l|\le N+1$) of the lower-order terms
one readily obtains that there exists at most one $b_j$ ($1\le j\le
l$) such that $\ds\left[\f{N+3}{2}\right]<|b_j|\le N+1$. Moreover,
$\ds\left[\f{N+3}{2}\right]\le N-2$ by $N\ge8$. Thus, applying
\eqref{2.8}-\eqref{2.10} and subsequently substituting
\eqref{2.21}-\eqref{2.23} into \eqref{2.20}, completes the
proof of Lemma~\ref{lem2.4}.
\end{proof}

Next we focus on the general time-weighted energy estimate of
$\p\G^a\psi$ with $0\le |a|\le N$ and $N\ge 8$.

\begin{lemma}[Time-weighted energy estimate of $\p \G^a\psi$ for
    $|a|\le N$]\label{lem2.5}
Let $\mu>0$ and $\la\in(0,1]$. Under \textup{assumption \eqref{2.8}},
we have that, for $t>0$,
\begin{multline}\label{2.24}
\sum_{0\leq|a|\leq N}\int_{\R^3}\biggl((1+t)^{2\la}|\p\G^a\psi|^2+|\G^a\psi|^2\biggr)dx
+C\sum_{0\leq|\al|\leq N}\int_0^t\int_{\R^3}(1+\tau)^\la|\p\G^a\psi|^2dxd\tau\\
 \leq  C\ve^2+C(1+K\ve)\int_0^tA(\tau)\sum_{0\leq|a|\leq N}\int_{\R^3}\biggl((1+\tau)^{2\la}|\p\G^a\psi|^2+\psi^2\biggr)dxd\tau,
\end{multline}
where the function $A$ has been defined in\/
\textup{Lemma~\ref{lem2.4}}.
\end{lemma}

\begin{proof}
Writing $\G^a=\t\G^b\p^c$ with $\t\G\in \{\O, S\}$, we will use
induction on $|b|$ to prove \eqref{2.24}. In view of Lemma~\ref{lem2.4},
it is enough to assume that $|c|=0$.

Suppose that \eqref{2.24} holds for $|b|\le l-1$, where $1\le l\le N$.
We then intend to establish \eqref{2.24} for $|b|=l$.

Acting with $\t\G^a$ (where $a=b$ and $|b|=l$) on both sides of \eqref{2.6}
yields
\begin{multline}\label{2.25}
  \square\t\G^a\psi+\ds\f{\mu}{(1+t)^\la}\p_t\t\G^a\psi
  +\f{2\la}{1+t}\p_t\t\G^a\psi =
  \sum_{|b_1|<|b|}\t\G^{b_1}\p^c\square\psi \\ +\t\G^aQ(\psi)
  -\left[\t\G^a,\f{\mu}{(1+t)^\la}\p_t\right]\psi-
  \left[\t\G^a,\f{2\la}{1+t}\p_t\right]\psi
  +\t\G^a\left((\la-1)(1+t)^{-2}\psi\right).
\end{multline}
Starting from \eqref{2.25}, as in the proof of Lemma~\ref{lem2.4}, we
can choose the multiplier $m(1+t)^{2\la}\p_t\t\G^a\psi
+(1+t)^{2\la-1}\t\G^a\psi+\kappa(1+t)^\la\t\G^a\psi$ to derive
\eqref{2.24}. Indeed, it follows from a direct computation that
\begin{multline}\label{2.26}
\sum_{\substack{|b|=l, \\ |c|\leq N-l}}\int_{\R^3}\left((1+t)^{2\la}|\p\t\G^b\p^c\psi|^2+|\t\G^b\p^c\psi|^2\right)dx
+C\sum_{\substack{|b|=l, \\ |c|\leq N-l}}\int_0^t\int_{\R^3}(1+\tau)^\la|\p\t\G^b\p^c\psi|^2dxd\tau \\
\begin{aligned}
& \leq  C\ve^2+C\sum_{\substack{|b_1|<l, \\ |c_1|\leq N-|b_1|}}\int_0^t\int_{\R^3}(1+\tau)^\la|\p\t\G^{b_1}\p^{c_1}\psi|^2dxd\tau \\
  &\qquad + C(1+K\ve)\int_0^tA(\tau)\sum_{\substack{|b_1|\leq l, \\ |c_1|\leq N-|b_1|}}\int_{\R^3}\left((1+\tau)^{2\la}|\p\t\G^{b_1}\p^{c_1}\psi|^2+|\t\G^{b_1}\p^{c_1}\psi|^2\right)dxd\tau \\
&\qquad +C\sum_{\substack{|b_1|\leq l, \\ |c_1|\leq N-|b_1|}}\int_0^t\left|\int_{\R^3}(1+\tau)^{2\la}(\p_t\t\G^a\psi)\t\G^{b_1}\p^{c_1}Q(\psi)dx\right| d\tau\\
&\qquad +C\sum_{\substack{|b_1|\leq l, \\ |c_1|\leq N-|b_1|}}\int_0^t\left|\int_{\R^3}(1+\tau)^\la(\t\G^a\psi)\t\G^{b_1}\p^{c_1}Q(\psi)dx\right|d\tau.
\end{aligned}
\end{multline}
Next we deal with the last two terms in the right-hand side of
\eqref{2.26}. Note that
\[
c^2(\rho)-1=-G(\psi)\int_0^1(c^2)'\left(-sG(\psi)\right)ds,
\]
where $G(\psi)=(1+t)^\la\p_t\psi+(1+t)^{\la-1}\psi
+(1+t)^{2\la}|\na\psi|^2/2+\mu\psi$. From this, it is readily seen
that the typical terms in $Q(\psi)$ are of the form $\psi\,\Dl\psi$,
$(1+t)^\la\p_t\na\psi\cdot\na\psi$, and
$(1+t)^{2\la}(\p_i\psi)(\p_j\psi)\p_{ij}\psi$. We analyze them
separately. Without loss of generality, we assume $|c_1|=0$ in the
last two terms of \eqref{2.26}; the treatment of the other cases is
easier.

\paragraph*{Part A} \textit{Estimates of $\ds\sum_{|b_1|\leq N }\int_0^t\left|\int_{\R^3}(1+\tau)^{2\la}(\p_t\t\G^a\psi)\t\G^{b_1}(\psi\Dl\psi)dx\right|d\tau$
and $\ds\sum_{|b_1|\leq N }\int_0^t\left|\int_{\R^3}(1+\tau)^{\la}\right.$
$\left.(\t\G^a\psi) \t\G^{b_1}(\psi\Dl\psi)dx\right|d\tau$.}

\medskip

Note that
\[
\t\G^{b_1}(\psi\Dl\psi)=\RN{1}_1+\RN{1}_2+\RN{1}_3,
\]
where
\begin{align*}
\RN{1}_1&=\psi\,\Dl\t\G^{b_1}\psi,\\
\RN{1}_2&=\ds\sum_{\substack{|b_1|=|b_2|+|b_3|, \\ 1\leq|b_2|\leq N-3}}(\t\G^{b_2}\psi)\Dl\t\G^{b_3}\psi,\\
\RN{1}_3&=\ds\sum_{\substack{|b_1|=|b_2|+|b_3|, \\ N-2\leq|b_2|\leq l}}(\t\G^{b_2}\psi)\Dl\t\G^{b_3}\psi.
\end{align*}
In view of  $b_1=a$ and
\begin{multline*}
(1+t)^{2\la}(\p_t\t\G^a\psi)\psi\Dl\t\G^a\psi \\
=\operatorname{div}\biggl((1+t)^{2\la}(\p_t\t\G^a\psi)\psi\na\t\G^a\psi\biggr)
+\f{1}{2}\p_t\biggl((1+t)^{2\la}|\na\t\G^a\psi|^2\psi\biggr)\\
 -(1+t)^{2\la}(\p_t\t\G^a\psi)\na\psi\cdot\na\t\G^a\psi
-\la(1+t)^{\la-1}|\na\t\G^a\psi|^2\psi-\f{1}{2}(1+t)^{2\la}|\na\t\G^a\psi|^2\p_t\psi,
\end{multline*}
we have by an integration by parts and \eqref{2.9}-\eqref{2.10}
\begin{multline}\label{2.27}
\int_0^t\left|\int_{\R^3}(1+\tau)^{2\la}(\p_t\t\G^a\psi) \RN{1}_1\,dx\right|d\tau
\le C\ve^2  +CK\ve\sum_{0\leq|a|\leq N}\int_{\R^3}(1+t)^{2\la}|\p\t\G^a\psi|^2dx \\
+CK\ve\sum_{0\leq|a|\leq N}\int_0^t\int_{\R^3}(1+\tau)^\la|\p\t\G^a\psi|^2dxd\tau.
\end{multline}
Moreover, it follows from \eqref{2.10} and \eqref{2.12} that
\begin{multline}\label{2.28}
\int_{\R^3} \left|(1+t)^{2\la}(\p_t\t\G^a\psi)(\t\G^{b_2}\psi)\Dl\t\G^{b_3}\psi\right|dx \\
\begin{aligned}
&\leq (1+t)^{2\la}\|\langle r-t\rangle^{-1}\t\G^{b_2}\psi\|_{L^{\infty}}\cdot\|\p_t\t\G^a\psi\|\cdot\|\langle r-t\rangle\Dl\t\G^{b_3}\psi\| \\
&\leq CK\ve(1+t)^\la\|\p_t\t\G^a\psi\|\sum_{|b_4|\leq |b_3|+1}\left(\|\na\t\G^{b_4}\psi\|+(1-\la)(1+t)^{-1}\|\t\G^{b_4}\psi\|\right) \\
&\leq CK\ve(1+t)^\la\|\p_t\t\G^a\psi\|\sum_{|b_4|\leq |b_3|+1}\|\na\t\G^{b_4}\psi\|+CK\ve(1+t)^\la\|\p_t\t\G^a\psi\|^2 \\
&\qquad +CK\ve(1-\la)(1+t)^{\la-2}\sum_{|b_4|\leq |b_3|+1}\|\t\G^{b_4}\psi\|^2.
\end{aligned}
\end{multline}
On the other hand, we have that by \eqref{2.9} and Hardy's inequality
\begin{multline}\label{2.29}
\int_{\R^3} \left|(1+t)^{2\la}(\p_t\t\G^a\psi)(\G^{b_2}\psi)\Dl\t\G^{b_3}\psi\right|\,dx \\
\begin{aligned}
&\leq (1+t)^{2\la}\|(1+r)\Dl\t\G^{b_3}\psi\|_{L^{\infty}}\cdot\|\p_t\t\G^a\psi\|\|(1+r)^{-1}\t\G^{b_2}\psi\| \\
&\leq CK\ve(1+t)^\la\|\p_t\t\G^a\psi\|\sum_{|b_4|\leq |b_2|}\|\na\t\G^{b_4}\psi\|.
\end{aligned}
\end{multline}
Combining \eqref{2.27}-\eqref{2.29} yields
\begin{multline}\label{2.30}
\sum_{|b_1|\leq N }\int_0^t\left|\int_{\R^3}(1+\tau)^{2\la}(\p_t\G^a\psi)\G^{b_1}(\psi\Dl\psi)dx\right|d\tau\\
\le  C\ve^2+CK\ve\sum_{0\leq|a|\leq N}\int_{\R^3}(1+t)^{2\la}|\p\t\G^a\psi|^2dx
+CK\ve\sum_{0\leq|a|\leq N}\int_0^t\int_{\R^3}(1+\tau)^\la|\p\t\G^a\psi|^2dxd\tau\\
 +CK\ve\ds\sum_{|b_4|\le N}\int_0^tA(\tau)\int_{\R^3}|\t\G^{b_4}\psi|^2dxd\tau.
\end{multline}
Note that
\begin{align*}
(1+t)^\la(\t\G^a\psi)\t\G^{b_1}(\psi\Dl\psi) &= \sum_{|b_2|+|b_3|=|b_1|}(1+t)^\la(\t\G^a\psi)(\t\G^{b_2}\psi)\Dl\t\G^{b_3}\psi \\
& =\ds \operatorname{div}\left(\sum_{|b_2|+|b_3|=|b_1|}(1+t)^\la(\t\G^a\psi)(\t\G^{b_2}\psi)\na\t\G^{b_3}\psi\right)+\sum_{i=4}^5\RN{1}_i,
\end{align*}
where
\begin{align*}
\RN{1}_4&=-\sum_{\substack{|b_2|\leq N-3,\\ |b_2|+|b_3|=|b_1|}}(1+t)^\la(\t\G^{b_2}\psi)(\na\t\G^a\psi)\cdot(\na\t\G^{b_3}\psi),\\
\RN{1}_5&= -\sum_{\substack{N-2\leq|b_2|\leq l-1,\\ |b_2|+|b_3|=|b_1|}}(1+t)^\la(\t\G^{b_2}\psi)(\na\t\G^a\psi)\cdot(\na\t\G^{b_3}\psi) \\
&\qquad -\sum_{|b_2|+|b_3|=|b_1|}(1+t)^\la(\t\G^a\psi)(\na\t\G^{b_2}\psi)\cdot(\na\t\G^{b_3}\psi).
\end{align*}
Therefore, by \eqref{2.10} and Hardy's inequality, we have
\begin{equation*}\label{2.31}
  \int_{\R^3}|\RN{1}_4|\,dx \leq CK\ve(1+t)^\la\,\|\na\t\G^a\psi\|\ds\sum_{|b_1|+3-N\le |b_3|\le N}\|\na\t\G^{b_3}\psi\|
\end{equation*}
and
\begin{equation*}\label{2.32}
\int_{\R^3}|\RN{1}_5|dx \leq CK\ve\|(1+r)^{-1}\t\G^{b_2}\psi\na\t\G^a\psi\|_{L^1}\leq CK\ve\|\na\t\G^{b_2}\psi\|\|\na\t\G^a\psi\|.
\end{equation*}
This yields
\begin{equation}\label{2.33}
\sum_{|b_1|\leq N }\int_0^t\left|\int_{\R^3}(1+\tau)^{\la}(\t\G^a\psi)
\t\G^{b_1}(\psi\Dl\psi)dx\right|d\tau\le
CK\ve\sum_{0\leq|a|\leq N}\int_0^t\int_{\R^3}(1+\tau)^\la|\p\t\G^a\psi|^2dxd\tau.
\end{equation}

\paragraph{Part B} \textit{Estimates of $\ds\sum_{|b_1|\leq N }\int_0^t\left|
  \int_{\R^3}(1+\tau)^{2\la}(\p_t\t\G^a\psi)\t\G^{b_1}\left((1+\tau)^{\la}
  \p_t\na\psi\cdot\na\psi\right)dx\right|d\tau$ and \linebreak
  $\ds\sum_{|b_1|\leq N
  }\int_0^t\left|\int_{\R^3}(1+\tau)^{\la}(\t\G^a\psi)
  \t\G^{b_1}\left((1+\tau)^{\la}\p_t\na\psi\cdot\na\psi\right)dx\right|d\tau$.}

One has
\begin{align*}
  \t\G^{b_1}\bigl((1+t)^\la\p_t\na\psi\cdot\na\psi\bigr)
  & =(1+t)^\la\p_t\na\t\G^{b_1}\psi\cdot\na\psi
  +\sum_{N-3\leq|b_2|\leq l-1}(1+t)^\la(\p_t\na\t\G^{b_2}\psi)\na\t\G^{b_3}\psi \\
&\qquad +\sum_{|b_2|\leq N-4}(1+t)^\la(\p_t\na\t\G^{b_2}\psi)\na\t\G^{b_3}\psi \\
&=\RN{2}_1+\RN{2}_2+\RN{2}_3.
\end{align*}
By \eqref{2.9}, we have
\begin{equation}\label{2.34}
  \sum_{|b_1|\leq N }\int_0^t\left|\int_{\R^3}(1+\tau)^{2\la}(\p_t\t\G^a\psi)\, \RN{2}_1\,dx
  \right|d\tau
\le CK\ve\sum_{0\leq|a|\leq N}\int_0^t\int_{\R^3}(1+\tau)^\la|\p\t\G^a\psi|^2dxd\tau.
\end{equation}
In addition, it follows from \eqref{2.9}, \eqref{2.11} and a direct computation that
\begin{multline}
(1+t)^{2\la}\|(\p_t\G^a\psi) \RN{2}_2\|_{L^1} \\
\begin{aligned}\label{2.35}
  &\leq (1+t)^{3\la}\ds\sum_{|b_2|\leq N-4}\|\langle r-t\rangle^{-1}\na\G^{b_3}\psi\|_{L^{\infty}}\cdot\|\p_t\G^a\psi\|\cdot\|\langle r-t\rangle\,\p_t\na\G^{b_2}\psi\| \\
&\leq CK\ve(1+t)^\la\|\p_t\G^a\psi\|\sum_{|c|\leq |b_2|+1}\left(\|\na\G^c\psi\|+(1-\la)(1+t)^{-1}\|\G^c\psi\|\right)\\
&\leq CK\ve(1+t)^\la\|\p_t\t\G^a\psi\|\sum_{|b_4|\leq |b_3|+1}\|\na\t\G^{b_4}\psi\|+CK\ve(1+t)^\la\|\p_t\t\G^a\psi\|^2 \\
&\qquad +CK\ve(1-\la)(1+t)^{\la-2}\sum_{|b_4|\leq |b_3|+1}\|\t\G^{b_4}\psi\|^2.
\end{aligned}
\end{multline}
Treating $\RN{2}_3$, we obtain by \eqref{2.10}
\begin{equation}\label{2.36}
  \int_0^t\left|\int_{\R^3}(1+\tau)^{2\la}(\p_t\t\G^a\psi) \RN{2}_3dx\right|d\tau\le
CK\ve\int_0^t\int_{\R^3}(1+\tau)^\la|\p\t\G^a\psi|^2\,dxd\tau.
\end{equation}
Collecting \eqref{2.34}-\eqref{2.36} yields
\begin{multline}\label{2.37}
\sum_{|b_1|\leq N}\int_0^t\left|\int_{\R^3}(1+\tau)^{2\la}
(\p_t\G^a\psi)\G^{b_1}\left((1+t)^{\la}
\p_t\na\psi\cdot\na\psi\right)dx\right|d\tau\\ \le CK\ve\sum_{0\leq|a|\leq
  N}\int_0^t\int_{\R^3}(1+\tau)^\la|\p\t\G^a\psi|^2\,dxd\tau +CK\ve
\sum_{|b_4|\le N}\int_0^tA(\tau)\int_{\R^3}|\t\G^{b_4}\psi|^2\,dxd\tau.
\end{multline}
In addition, one notes that
\begin{multline*}
2(1+t)^{2\la}(\t\G^a\psi)\t\G^a\left(\p_t\na\psi\cdot\na\psi\right) =
\sum_{|c|\leq|a|}\p_t\left((1+t)^{2\la}\t\G^a\psi\G^c(|\na\psi|^2)\right) \\
-2\la(1+t)^{2\la-1}(\t\G^a\psi)\t\G^c\left(|\na\psi|^2\right)
-(1+t)^{2\la}(\p_t\t\G^a\psi)\t\G^c\left(|\na\psi|^2\right).
\end{multline*}
From this and \eqref{2.10}, we have
\begin{multline}\label{2.38}
\sum_{|b_1|\leq N }\int_0^t\left|\int_{\R^3}(1+\tau)^{\la}(\t\G^a\psi)
\t\G^{b_1}\left((1+\tau)^{\la}\p_t\na\psi\cdot\na\psi\right)dx\right|\,d\tau \le
C\ve^2 \\ +CK\ve\sum_{0\leq|a|\leq
  N}\int_{\R^3}(1+t)^{2\la}|\p\t\G^a\psi|^2\,dx
+CK\ve\sum_{0\leq|a|\leq
  N}\int_0^t\int_{\R^3}(1+\tau)^\la|\p\t\G^a\psi|^2\,dxd\tau.
\end{multline}

\paragraph{Part C} \textit{Estimates of $\ds\sum_{|b_1|\leq N }\int_0^t
  \left|\int_{\R^3}(1+\tau)^{2\la}(\p_t\t\G^a\psi)\t\G^{b_1}\left((1+\tau)^{2\la}(\p_i\psi)
  (\p_j\psi)\p_{ij}\psi\right)dx\right|d\tau$
  and \linebreak $\ds\sum_{|b_1|\leq N
  }\int_0^t\left|\int_{\R^3}(1+\tau)^{\la}(\t\G^a\psi)
  \t\G^{b_1}\left((1+\tau)^{2\la}(\p_i\psi)(\p_j\psi)\p_{ij}\psi\right)dx\right|d\tau$.}

\medskip
A direct computation yields
\begin{align*}
  \t\G^{b_1}((\p_i\psi)(\p_j\psi)\p_{ij}\psi) &=
  \f{1}{2}\,\t\G^{b_1}\left((\p_i\psi)\p_i\left(|\na\psi|^2\right)\right)\\
  & =\f{1}{2}\,(\p_i\psi)\p_i\t\G^{b_1}\left(|\na\psi|^2\right)
  +\sum_{N-3\leq|b_2|\leq |b_1|-1}(\na^2\t\G^{b_2}\psi)(\na\t\G^{b_3}\psi)\na\t\G^{b_4}\psi \\
  & \qquad
  + \sum_{|b_2|\leq N-4}(\na^2\t\G^{b_2}\psi)(\na\t\G^{b_3}\psi)\na\t\G^{b_4}\psi\\
&=\RN{3}_1+\RN{3}_2+\RN{3}_3.
\end{align*}
As in the treatment of $\RN{2}_1$ in Part B, we have
\begin{equation}\label{2.39}
  \sum_{|b_1|\leq N }\int_0^t\left|\int_{\R^3}(1+\tau)^{2\la}(\p_t\t\G^a\psi)\, \RN{3}_1dx
  \right|d\tau
\le CK\ve\sum_{0\leq|a|\leq N}\int_0^t\int_{\R^3}(1+\tau)^\la|\p\t\G^a\psi|^2dxd\tau.
\end{equation}
By \eqref{2.9} and \eqref{2.12}, for the term $\RN{3}_2$, we have
\begin{multline}\label{2.40}
  (1+t)^{4\la}\|(\p_t\t\G^a\psi)(\na^2\t\G^{b_2}\psi)(\na\t\G^{b_3}\psi)
  \na\t\G^{b_4}\psi\|_{L^1} \\
\begin{aligned}
&\leq (1+t)^{4\la}\|\langle r-t\rangle^{-1}(\na\t\G^{b_3}\psi)\na\t\G^{b_4}\psi\|_{L^{\infty}}\cdot\|\p_t\t\G^a\psi\|\cdot\|\langle r-t\rangle\na^2\t\G^{b_2}\psi\|\\
&\leq CK\ve(1+t)^\la\|\p_t\t\G^a\psi\|\sum_{|b_4|\leq |b_3|+1}\|\na\t\G^{b_4}\psi\|+CK\ve(1+t)^\la\|\p_t\t\G^a\psi\|^2 \\
&\qquad +\, CK\ve(1-\la)(1+t)^{\la-2}\sum_{|b_4|\leq |b_3|+1}\|\t\G^{b_4}\psi\|^2.
\end{aligned}
\end{multline}
By \eqref{2.10} and \eqref{2.11}, for the term $\RN{3}_3$, one has
\begin{equation}\label{2.41}
  (1+t)^{4\la}\|(\p_t\t\G^a\psi)(\na^2\t\G^{b_2}\psi)(\na\t\G^{b_3}\psi)
  \na\t\G^{b_4}\psi\|_{L^1} \leq CK\ve(1+t)^\la\|\p_t\t\G^a\psi\|\sum_{|c|\leq|b_1|}
  \|\na\t\G^c\psi\|.
\end{equation}
Collecting \eqref{2.39}-\eqref{2.41} yields
\begin{multline}\label{2.42}
  \sum_{|b_1|\leq N }\int_0^t\left|\int_{\R^3}(1+\tau)^{2\la}(\p_t\t\G^a\psi)\t\G^{b_1}
  \left((1+\tau)^{2\la}(\p_i\psi)(\p_j\psi)\p_{ij}\psi\right)dx\right|d\tau\\
\le CK\ve\sum_{0\leq|a|\leq N}\int_0^t\int_{\R^3}(1+\tau)^\la|\p\t\G^a\psi|^2dxd\tau
+CK\ve\sum_{|b_4|\le N}\int_0^tA(\tau)\int_{\R^3}|\t\G^{b_4}\psi|^2dxd\tau.
\end{multline}
In addition,
\begin{multline*}
 2(1+t)^{3\la}(\G^a\psi)\G^{b_1}\left((\p_i\psi)(\p_j\psi)\p_{ij}\psi\right) \\
 =\operatorname{div}\left((1+t)^{3\la}(\G^a\psi)(\na\psi)
 \G^{b_1}\left(|\na\psi|^2\right)\right)
 -(1+t)^{3\la}(\na\G^a\psi)(\na\psi)\G^{b_1}\left(|\na\psi|^2\right) \\
 -(1+t)^{3\la}(\G^a\psi)(\Dl\psi)\G^{b_1}\left(|\na\psi|^2\right)
 +\sum_{|b_2|\leq |b_1|-1}(1+t)^{3\la}(\G^a\psi)(\na^2\G^{b_2}\psi)(\na\G^{b_3}\psi)\na
 \G^{b_4}\left(|\psi|^2\right).
\end{multline*}
Together with \eqref{2.9}-\eqref{2.10} this yields
\begin{multline}\label{2.43}
\sum_{|b_1|\leq N }\left|\int_0^t\int_{\R^3}(1+\tau)^{\la}(\G^a\psi)
\G^{b_1}\left((1+\tau)^{2\la}(\p_i\psi)(\p_j\psi)\p_{ij}\psi\right)dxd\tau\right|\\
\le CK\ve\ds\sum_{|a|\leq N }\int_0^t\int_{\R^3}(1+\tau)^\la|\p\G^a\psi|^2dxd\tau.
\end{multline}
Therefore, substituting \eqref{2.30}, \eqref{2.33},
\eqref{2.37}-\eqref{2.38}, and \eqref{2.42}-\eqref{2.43} into
\eqref{2.26} and utilizing the smallness of $\ve>0$ gives
\eqref{2.24}.
\end{proof}

Based on Lemmas \ref{lem2.4}--\ref{lem2.5}, we now prove
Theorem~\ref{thm1.1}.

\begin{proof}[Proof of Theorem~\ref{thm1.1}]
By Lemma~\ref{lem2.4} and Lemma~\ref{lem2.5}, one has that, for fixed
$N\ge 9$,
\[
E_N(t) \leq C\ve^2+C(1+K\ve)\int_0^t A(t')E_N(t')\,dt'.
\]
Choosing the constants $K>0$ large and $\ve>0$ small, by Gronwall's
inequality one gets that, for any $t\ge 0$,
\[
E_N(t)\le e^{C(1+K\ve)\|A(t)\|_{L^1}}E_N(0)\leq \ds\f{1}{2}K^2\ve^2
\]
Thus, Theorem~\ref{thm1.1} is proved by the assumption that $E_N(t)\leq
K^2\ve^2$ and a continuous induction argument.
\end{proof}


\section{Blowup for small data in case \boldmath $\la>1$}
\label{chap3}

In this section, we shall prove the blowup result of
Theorem~\ref{thm1.2} which is valid in case $\la>1$.

\begin{proof}[Proof of Theorem~\ref{thm1.2}]
  We divide the proof into two cases.

\subsubsection*{Case 1: \boldmath $\g=2$.}

Let $(\rho,u)$ be a smooth solution of \eqref{1.1}. For $l>0$, we define
\begin{equation}\label{3.1}
P(t,l)=\int_{|x|>l}\eta(x,l)\left(\rho(t,x)-\bar\rho\right)dx,
\end{equation}
where
\begin{equation*}\label{3.2}
\eta(x,l)=|x|^{-1}(|x|-l)^2.
\end{equation*}
Employing the first equation in \eqref{1.1} and an integration by
parts, we see that
\begin{align*}
  \p_tP(t,l)&=\int_{|x|>l}\eta(x,l)\p_t\left(\rho(t,x)-\bar\rho\right)dx=
  -\,\int_{|x|>l}\eta(x,l)\operatorname{div}(\rho u)(t,x)\,dx \\
&=\int_{|x|>l}(\na_x\eta)(x,l)\cdot(\rho u)(t,x)\,dx,
\end{align*}
where we have used the fact that $\eta(x,l)=0$ on $|x|=l$ and that
$u(t,x)=0$ for $|x|\geq t+M$.

We first show that, under the assumptions \eqref{1.11}-\eqref{1.12},
$P(t,l)$ defined by \eqref{3.1} is nonnegative for $l\geq M_0$. By
differentiating $\p_tP(t,l)$ again and using the second equation in
\eqref{1.1}, we find that
\begin{multline}\label{3.3}
\p_t^2P(t,l) =\int_{|x|>l}(\na_x\eta)(x,l)\cdot\p_t(\rho u)(t,x)\,dx
 =-\sum_{i,j}\int_{|x|>l}(\p_{x_i}\eta)\,\p_{x_j}(\rho
  u_iu_j)\,dx \\ -\int_{|x|>l}(\na_x\eta)(x,l)\cdot\na
p\,dx -  \f{\mu}{(1+t)^\la}\int_{|x|>l}(\na_x\eta)(x,l)\cdot(\rho
u)(t,x)\,dx,
\end{multline}
where $\na_x\eta(x,l)=|x|^{-3}(|x|^2-l^2)x$ that vanishes on
$|x|=l$. Integration by parts implies that
\begin{align*}
  \p_t^2P(t,l)+ \f{\mu}{(1+t)^\la}\,\p_tP(t,l)&=
  \sum_{i,j}\int_{|x|>l}(\p_{x_ix_j}\eta)\rho
  u_iu_j\,dx+\int_{|x|>l}(\Dl\eta) p\,dx \\ &\equiv
  J_1(t,l)+\int_{|x|>l}\left(\Dl\eta\right) p\,dx.
\end{align*}
A direct computation of $\p_{x_ix_j}\eta$ shows that
\begin{multline}\label{3.4}
J_1(t,l)=\int_{|x|>l}\f{2l^2}{|x|^3}\,\rho\left(\f{x}{|x|}\cdot u\right)^2dx \\
-\int_{|x|>l}\f{|x|^2-l^2}{|x|^3}\,\rho\left(\f{x}{|x|}\cdot u\right)^2dx
 +\int_{|x|>l}\f{|x|^2-l^2}{|x|^3}\,\rho|u|^2dx \geq 0.
\end{multline}
Together with the fact that $\Dl\eta=2|x|^{-1}\geq0$, this yields
\[
\p_t^2P(t,l)+\f{\mu}{(1+t)^\la}\,\p_tP(t,l)\geq0.
\]
Note that assumptions \eqref{1.11} and \eqref{1.12} imply $P(0,l)>0$
and $\p_tP(0,l)\geq0$ for $M_0\le l\le M$. Integrating the above
differential inequality twice yields $P(t,l)\geq0$ for $l\ge M_0$ and
$t\geq0$.

Next we derive a lower bound of $P(t,l)$. We now rewrite \eqref{3.3}
as
\begin{align*}
\p_t^2P(t,l)&=-\sum_{i,j}\int_{|x|>l}(\p_{x_i}\eta)\,\p_{x_j}(\rho
  u_iu_j)\,dx-\int_{|x|>l}(\na_x\eta)(x,l)\cdot\na(p-\bar p)\,dx
\\ &\qquad- \f{\mu}{(1+t)^\la}\int_{|x|>l}(\na_x\eta)(x,l)\cdot(\rho
u)(t,x)\,dx,
\end{align*}
where $\bar p=p(\bar\rho)$. Let $J_2(t,l)\equiv
\int_{|x|>l}(\Dl\eta)(p-\bar p)\,dx$. Then the above equation reads
\begin{equation}\label{3.5}
\p_t^2P(t,l)+\f{\mu}{(1+t)^\la}\,\p_tP(t,l)=J_1(t,l)+J_2(t,l).
\end{equation}
Note that
\begin{equation*}
\p_l^2\eta(x,l)=2|x|^{-1}=\Dl_x\eta(x,l).
\end{equation*}
Then
\begin{equation}\label{3.6}
J_2(t,l)=\int_{|x|>l}\p_l^2\eta(x,l)(p(t,x)-\bar p)\,dx=
\p_l^2\int_{|x|>l}\eta(x,l)(p(t,x)-\bar p)\,dx,
\end{equation}
where we have used the fact that $\eta$ and $\p_l\eta$ vanish on
$|x|=l$. Combining \eqref{3.4}-\eqref{3.6}, we arrive at
\begin{equation}\label{3.7}
\p_t^2P(t,l)-\p_l^2P(t,l)+\f{\mu}{(1+t)^\la}\,\p_tP(t,l)\geq G(t,l),
\end{equation}
where
\begin{equation}\label{3.8}
G(t,l)=\p_l^2\int_{|x|>l}\eta(x,l)\left(p-\bar p-(\rho-\bar\rho)\right)dx
=\int_{|x|>l}2|x|^{-1}\left(p-\bar p-(\rho-\bar\rho)\right)dx.
\end{equation}
Thanks to $\g=2$ and the sound speed $\bar c=\sqrt{2A\bar\rho}=1$, we
have
\begin{equation}\label{3.9}
p-\bar p-(\rho-\bar\rho)=A\left(\rho^2-\bar\rho^2
-2\bar\rho\left(\rho-\bar\rho\right)\right)=A(\rho-\bar\rho)^2.
\end{equation}
Substituting \eqref{3.9} into \eqref{3.8} gives
\begin{equation*}
G(t,l)\geq 0.
\end{equation*}
To estimate $P=P(t,l)$ from inequality \eqref{3.7}, we first study the
solution of the equation
\begin{equation*}
\p_t^2\t P(t,l)-\p_l^2\t P(t,l)+\f{\mu}{(1+t)^\la}\,\p_t\t P(t,l)=G(t,l),
\end{equation*}
where $\t P(0,l)=P(0,l)$ and $\p_t\t P(0,l)=\p_t P(0,l)$.

Rewriting the above equation as
\begin{equation*}
\p_t^2\t P(t,l)-\p_l^2\t P(t,l)+ \f{\mu}{(1+t)^\la}\left(\p_t\t
P(t,l)-\p_l\t P(t,l)\right) =G(t,l)- \f{\mu}{(1+t)^\la}\,\p_l\t P(t,l),
\end{equation*}
by the method of characteristics we have, for $l\ge t\ge0$,
\begin{align*}
\t P(t,l)&=\f{1}{2}\,P(0,l+t)+\f{1}{2\beta(t)}\,P(0,l-t)
+\f{1}{2}\int_0^t\f1{\beta(\tau)}\,\f{\mu}{(1+\tau)^\la}\,P(0,l+t-2\tau)\,d\tau
\\ &\qquad +\int_0^t\f1{\beta(\tau)}\,\p_tP(0,l+t-2\tau)\,d\tau
+\f{1}{2}\int_0^t\int_{l-t+\tau}^{l+t-\tau}
\f{\beta(\tau)}{\beta\left(\f{l+t+\tau-y}{2}
  \right)}\,G(\tau,y)\,dyd\tau \\ &\qquad
+\f{1}{2}\int_0^t\f{\beta(\tau)}{\beta(t)}\,\f{\mu}{(1+\tau)^\la}\, \t
P(\tau,l-t+\tau)\,d\tau \\ &\qquad +\f{1}{2}\int_0^t\int_\tau^t
\f{\beta(\tau)}{\beta(s)}\, \, \f{\mu^2}{(1+\tau)^\la(1+s)^\la}\,\t
P(\tau,l+t-2s+\tau)\,dsd\tau \\
&\qquad -\f{1}{2}\int_0^t\ds\f{\mu}{(1+\tau)^\la}\,\t P(\tau,l+t-\tau)\,d\tau,
\end{align*}
see \eqref{beta}. Together with assumptions
\eqref{1.11}-\eqref{1.12} this yields, for $l\geq t+M_0$,
\begin{multline}\label{3.10}
P(t,l) \ge \t P(t,l)\ge \f{1}{2\beta(t)}\,q_0(l-t)
\\ +\f{1}{2}\int_0^t\int_{l-t+\tau}^{l+t-\tau}
\frac{\beta(\tau)}{\beta\left(\f{l+t+\tau-y}{2} \right)}\,
G(\tau,y)\,dyd\tau -\f{1}{2}\int_0^t
\f{\mu}{(1+\tau)^\la}\,P(\tau,l+t-\tau)d\tau.
\end{multline}

Define the function
\begin{equation}\label{3.11}
F(t)\equiv\int_0^t(t-\tau)\int_{\tau+M_0}^{\tau+M}P(\tau, l)\,\f{dl}l d\tau.
\end{equation}
Then, by \eqref{3.10},  we have that
\begin{multline}\label{3.12}
F''(t)=\int_{t+M_0}^{t+M}P(t,l)\,\f{dl}l
\geq\f{1}{2\beta(t)}\int_{t+M_0}^{t+M}q_0(l-t)\,\f{dl}l
\\ +\f{1}{2}\int_{t+M_0}^{t+M}\int_0^t\int_{l-t+\tau}^{l+t-\tau}
\frac{\beta(\tau)}{\beta\left(\f{l+t+\tau-y}{2} \right)}\, G(\tau,
y)\,dyd\tau \f{dl}l
\\ -\f{1}{2}\int_{t+M_0}^{t+M}\int_0^t\ds\f{\mu}{(1+\tau)^\la}\,P(\tau,
l+t-\tau)\,d\tau \f{dl}l \equiv J_3+J_4-J_5.
\end{multline}
From $\la>1$ and assumption \eqref{1.11}, we see that
\begin{equation}\label{3.13}
J_3\geq
\f{c_1}{t+M}\int_{t+M_0}^{t+M}q_0(l-t)\,dl=\f{c_1}{t+M}\int_{M_0}^Mq_0(l)
\,dl=\f{c_2\ve}{t+M}
\end{equation}
where $c_1,c_2>0$ are constants independent of $\ve$. Note that
$P(\tau,y)$ is supported in $\{y\colon y\leq \tau+M\}$ and nonnegative
for $y\geq M_0$. Hence, there exists a constant $C_1>0$ such that
\begin{equation}\label{3.14}
J_5\leq\f{C_1}{(1+t)^\la}\int_0^t\int_{\tau+M_0}^{\tau+M}P(\tau,y)\,\f{dy}yd\tau
=\f{C_1}{(1+t)^\la}F'(t).
\end{equation}
Substituting \eqref{3.14} into \eqref{3.12} gives
\begin{equation}\label{3.15}
F''(t)+\f{C_1}{(1+t)^\la}F'(t)\geq J_3+J_4.
\end{equation}
To bound $J_4$ from below, we write
\begin{align}
  J_4&=\f{1}{2}\int_0^{t-M_1}\int_{\tau+M_0}^{\tau+M}G(\tau,y)
  \int_{t+M_0}^{y+t-\tau}\frac{\beta(\tau)}{\beta\left(\f{l+t+\tau-y}{2}
    \right)}\, \f{dl}ldyd\tau \notag
  \\ &\quad+\f{1}{2}\int_{t-M_1}^t\int_{\tau+M_0}^{2t-\tau+M_0}G(\tau,y)
  \int_{t+M_0}^{y+t-\tau}\frac{\beta(\tau)}{\beta\left(\f{l+t+\tau-y}{2}
    \right)}\, \f{dl}ldyd\tau \notag
  \\ &\quad+\f{1}{2}\int_{t-M_1}^t\int_{2t-\tau+M_0}^{\tau+M}G(\tau,y)
  \int_{y-t+\tau}^{y+t-\tau}\frac{\beta(\tau)}{\beta\left(\f{l+t+\tau-y}{2}
    \right)}\, \f{dl}ldyd\tau \notag \\ &\equiv J_{4,1}+J_{4,2}+J_{4,3},
  \label{3.16}
\end{align}
where $M_1=\left(M-M_0\right)/2$. For $t<M_1$, $t-M_1$ in the limits
of integration is replaced by $0$. By $\la>1$, for the integrand in
$J_{4,1}$ we have that
\begin{equation}\label{3.17}
  \int_{t+M_0}^{y+t-\tau}
  \frac{\beta(\tau)}{\beta\left(\f{l+t+\tau-y}{2} \right)}\, \f{dl}l
  \geq c\,\f{y-\tau-M_0}{t+M}
  \geq c\,\f{(t-\tau)(y-\tau-M_0)^2}{(t+M)^2}.
\end{equation}
Analogously, for the integrands in $J_{4,2}$ and $J_{4,3}$ we have
that
\begin{equation}\label{3.18}
\int_{t+M_0}^{y+t-\tau}\frac{\beta(\tau)}{\beta\left(\f{l+t+\tau-y}{2}
  \right)}\, \f{dl}l \geq c\,\f{(t-\tau)(y-\tau-M_0)^2}{(t+M)^2}
\end{equation}
and
\begin{equation}\label{3.19}
  \int_{y-t+\tau}^{y+t-\tau}\frac{\beta(\tau)}{\beta\left(\f{l+t+\tau-y}{2}
    \right)}\, \f{dl}l \geq c\,\f{t-\tau}{t+M} \geq
  c\,\f{(t-\tau)(y-\tau-M_0)^2}{(t+M)^2},
\end{equation}
where $c>0$ is a constant. Substituting \eqref{3.17}-\eqref{3.19} into
\eqref{3.16} yields
\begin{align*}
J_4\geq
\f{c}{(t+M)^2}\int_0^t(t-\tau)\int_{\tau+M_0}^{\tau+M}(y-\tau-M_0)^2
\p_y^2\t G(\tau,y)\,dyd\tau,
\end{align*}
where $\t G(t,l)=\int_{|x|>l}\eta(x,l)\left(p-\bar
p-(\rho-\bar\rho)\right)dx$.  Note that $\t G(\tau,y)=\p_y\t
G(\tau,y)=0$ for $y=\tau+M$. Thus, it follows from the integration by
parts together with \eqref{3.8}-\eqref{3.9} that
\begin{align}
J_4&\geq \f{c}{(t+M)^2}\int_0^t(t-\tau)\int_{\tau+M_0}^{\tau+M}\t
G(\tau,y)\,dyd\tau \notag \\ &\geq
\f{c}{(t+M)^2}\int_0^t(t-\tau)\int_{\tau+M_0}^{\tau+M}
\int_{|x|>y}\eta(x,y)\left(\rho(\tau,x)-\bar\rho\right)^2dxdyd\tau \notag
\\ &\equiv \f{c}{(t+M)^2}\,J_6. \label{3.20}
\end{align}
By applying the Cauchy-Schwartz inequality to $F(t)$ defined by
\eqref{3.11}, we arrive at
\begin{equation}\label{3.21}
  F^2(t) \leq J_6\int_0^t(t-\tau)\int_{\tau+M_0}^{\tau+M}
  \int_{y<|x|<\tau+M}\eta(x,y)\,dx\f{dy}{y^2}d\tau\equiv J_6J_7.
\end{equation}
We estimate $J_7$ as
\begin{align}
  J_7 &= \int_0^t(t-\tau)\int_{\tau+M_0}^{\tau+M}
  \int_{y<|x|<\tau+M}\f{(|x|-y)^2}{|x|}\,dx\f{dy}{y^2}d\tau \notag
  \\ &= \int_0^t(t-\tau)\int_{\tau+M_0}^{\tau+M}\int_y^{\tau+M}4\pi
  l\left(l-y\right)^2 dl\f{dy}{y^2}d\tau \notag \\ &\leq
  C\int_0^t(t-\tau)\int_{\tau+M_0}^{\tau+M}(\tau+M)\left(\tau+M-y\right)^3
  \f{dy}{y^2}d\tau \notag \\ & \leq
  C\int_0^t(t-\tau)(\tau+M)\int_{\tau+M_0}^{\tau+M}\f{dy}{y^2}d\tau
  \notag \\ &\leq C\int_0^t\f{t-\tau}{\tau+M}\,d\tau \leq
  C\left(t+M\right)\log\f{t}{M+1}. \label{3.22}
\end{align}
Combining \eqref{3.13}, \eqref{3.15} and \eqref{3.20}-\eqref{3.22}
gives the ordinary differential inequalities
\begin{align}
F''(t)+\f{C_1}{(1+t)^\la}\,F'(t) &\geq \f{c_2\ve}{t+M},
&& \hspace{-15mm}t\geq0, \hspace{15mm} \label{3.23}
\\ F''(t)+\f{C_1}{(1+t)^\la}\,F'(t) &\geq
C\left(t+M\right)^3\log\f{t}{M+1}\,F^2(t),
&& \hspace{-15mm}t\geq0. \hspace{15mm} \label{3.24}
\end{align}
Next, we apply \eqref{3.23}-\eqref{3.24} to prove that the lifespan $T_\ve$ of smooth solution $F(t)$
is finite for all $0<\ve\leq\ve_0$.
The fact that $F(0)=F'(0)=0$, together with \eqref{3.23}, yields
\begin{align}
F'(t) &\geq C\ve\log(t/M+1), \qquad t\geq0, \label{3.25} \\
F(t)  &\geq C\ve(t+M)\log(t/M+1), \quad t\geq t_1\equiv Me^2, \label{3.26}
\end{align}
where the constant $C>0$ is independent of $\ve$.
Substituting \eqref{3.26} into \eqref{3.24} derives
\begin{equation*}\label{3.27}
F''(t)+\f{C_1}{(1+t)^\la}F'(t) \geq C\ve^2(t+M)^{-1}\log(t/M+1), \quad t\geq t_1,
\end{equation*}
which leads to the improvement
\begin{equation}\label{3.28}
F(t) \geq C\ve^2(t+M)\log^2(t/M+1), \quad t\geq t_2\equiv Me^3>t_1.
\end{equation}
Substituting this into \eqref{3.24} derives
\begin{equation}\label{3.29}
F''(t)+\f{C_1}{(1+t)^\la}F'(t) \geq C\ve^2(t+M)^{-2}\log(t/M+1)F(t), \qquad t\geq t_2.
\end{equation}
It follows from \eqref{3.25} that $F'(t)\geq0$ for $t\geq0$. Then
multiplying \eqref{3.29} by $F'(t)$ and integrating from $t_3$ (which
will be chosen later) to $t$ yield
\begin{equation*}
F'(t)^2 \geq C_2F'(t_3)^2+C_3\ve^2\int_{t_3}^t(s+M)^{-2}\log(s/M+1)[F(s)^2]'ds.
\end{equation*}
Integrating by parts yields
\begin{multline}\label{3.30}
F'(t)^2 \geq
C_2F'(t_3)^2+C_3\ve^2\left(t+M)^{-2}\log(t/M+1)F(t)^2-(t_3+M)^{-2}
\log(t_3/M+1)F(t_3)^2\right)
\\ -\int_{t_3}^t\left(\f{\log(s/M+1)}{(s+M)^2}\right)'F(s)^2\,ds, \quad  t\geq t_3,
\end{multline}
where $\left(\f{\log(s/M+1)}{(s+M)^2}\right)'\leq0$ for $t\geq t_3\geq
t_2$. On the other hand, \eqref{3.23} implies
\begin{equation*}
\left(e^{-\f{C_1}{\la-1}[(1+t)^{1-\la}-1]}F'(t)\right)' \geq 0, \quad t\geq 0,
\end{equation*}
which yields for $0\leq t\leq \tau$
\begin{equation}\label{3.31}
F'(t) \leq e^{\f{C_1}{\la-1}[(1+t)^{1-\la}-(1+\tau)^{1-\la}]}F'(\tau).
\end{equation}
Together with $F(0)=0$, this yields
\begin{equation}\label{3.32}
F(t)=\int_0^tF'(s)ds \leq C_4tF'(t), \quad t>0.
\end{equation}
Choose
\begin{equation}\label{3.33}
t_3=M\left(e^\f{C_2}{2C_3C_4\ve^2}-1\right)
\end{equation}
which satisfies $2C_3C_4\log(t_3/M+1)\ve^2=C_2$. Together with
\eqref{3.30} and \eqref{3.32}, this yields
\begin{equation}\label{3.34}
F'(t) \geq \sqrt C_3\ve(t+M)^{-1}\log^\f{1}{2}(t/M+1)F(t), \quad t\geq t_3.
\end{equation}
By integrating \eqref{3.34} from $t_3$ to $t$, we arrive at
\begin{equation*}
\log\f{F(t)}{F(t_3)} \geq C\ve\log^\f{3}{2}(\f{t+M}{t_3+M}), \quad t\geq t_3.
\end{equation*}
If $t\geq t_4\equiv Ct_3^2$, then we have
\begin{equation*}
\log\f{F(t)}{F(t_3)} \geq 8\log(t/M+1).
\end{equation*}
Together with \eqref{3.28} for $F(t_3)$, this yields
\begin{equation}\label{3.35}
F(t) \geq C\ve^2(t+M)^8, \quad t\geq t_4.
\end{equation}
Substituting this into \eqref{3.24} derives
\begin{equation*}
F''(t)+\f{C_1}{(1+t)^\la}F'(t) \geq C\ve F(t)^\f{3}{2}, \quad t\geq
t_4.
\end{equation*}
Multiplying this differential inequality by $F'(t)$ and integrating
from $t_4$ to $t$ yields
\begin{equation*}
F'(t)^2 \geq C\ve\left(F(t)^\f{5}{2}-F(t_4)^\f{5}{2}\right).
\end{equation*}
On the other hand, \eqref{3.31} and \eqref{3.32} imply that, for $t\geq t_4$,
\begin{equation*}
F(t)=F'(\xi)(t-t_4)+F(t_4) \geq CF'(t_4)(t-t_4) \geq CF(t_4)\f{t-t_4}{t_4},
\end{equation*}
where $t_4\leq \xi \leq t$. If $t\geq t_5\equiv Ct_4$, then we have
\begin{equation*}
F(t)^\f{5}{2}-F(t_4)^\f{5}{2} \geq \f{1}{2}F(t)^\f{5}{2}.
\end{equation*}
Thus
\begin{equation}\label{3.36}
F'(t) \geq C\sqrt\ve F(t)^\f{5}{4}, \quad t\geq t_5.
\end{equation}
If $T_\ve>2t_5$, then integrating \eqref{3.36} from $t_5$ to $T_\ve$ derives
\begin{equation*}
F(t_5)^{-\f{1}{4}}-F(T_\ve)^{-\f{1}{4}} \geq C\sqrt\ve T_\ve.
\end{equation*}
We see from \eqref{3.35} and $t_5=Ct_3^2$ that
\begin{equation*}
F(t_5)\geq C\ve^2e^\f{C}{\ve^2},
\end{equation*}
which together with $F(T_\ve)>0$ is a contradiction. Thus, $T_\ve\leq
2t_5=Ct_3^2$. From the choice of $t_3$ in \eqref{3.33}, we see that
$T_\ve\leq e^{C/\ve^2}$.


\subsubsection*{Case 2: \boldmath $\g>1$ and $\g\not=2$.}

Recall that the sound speed is $\bar c=\sqrt{\g
  A\bar\rho^{\g-1}}=1$. Instead of \eqref{3.9} we have
\begin{equation*}\label{3.37}
p-\bar p-(\rho-\bar\rho)=
A\left(\rho^\g-\bar\rho^\g-\g\bar\rho^{\g-1}(\rho-\bar\rho)\right)\equiv
A\psi(\rho,\bar\rho).
\end{equation*}
The convexity of $\rho^\g$ for $\g>1$ implies that
$\psi(\rho,\bar\rho)$ is positive for $\rho\neq\bar\rho$. Applying
Taylor's theorem, we have
\begin{equation*}
\psi(\rho,\bar\rho) \geq C(\g,\bar\rho)\,\Phi_\g(\rho,\bar\rho),
\end{equation*}
where $C(\g,\bar\rho)$ is a positive constant and $\Phi_\g$ is given by
\begin{equation*}
\Phi_\g(\rho,\bar\rho)=
\begin{cases}
(\bar\rho-\rho)^\g, & \rho< \f{1}{2}\bar\rho,\\
(\rho-\bar\rho)^2, & \f{1}{2}\bar\rho\leq\rho\leq2\bar\rho,\\
(\rho-\bar\rho)^\g, & \rho>2\bar\rho.\\
\end{cases}
\end{equation*}
For $\g>2$, we have that
$(\bar\rho-\rho)^\g=(\bar\rho-\rho)^2(\bar\rho-\rho)^{\g-2}\geq
C(\g,\bar\rho)(\rho-\bar\rho)^2$ for $2\rho<\bar\rho$ and
$(\rho-\bar\rho)^\g=(\rho-\bar\rho)^2(\rho-\bar\rho)^{\g-2}\geq
C(\g,\bar\rho)(\rho-\bar\rho)^2$ for $\rho> 2\bar\rho$.  Thus,
$\psi(\rho,\bar\rho) \geq C(\g,\bar\rho)(\rho-\bar\rho)^2$. In this
case, Theorem~\ref{thm1.2} can be shown completely analogously to
Case~1.

Next we treat the case $1<\g<2$. We define $F(t)$ as in \eqref{3.11},
\begin{equation*}
F(t)=\int_0^t\int_{\tau+M_0}^{\tau+M}\f{1}{l}
\int_{|x|>l}\f{(|x|-l)^2}{|x|}\left(\rho(\tau,x)-\bar\rho\right)dxdl d\tau.
\end{equation*}
Similarly to the case of $\g=2$, we have
\begin{equation}\label{3.38}
F''(t)\geq J_3+J_4-J_5,
\end{equation}
where
\begin{align*}
J_3 &\geq \f{C\ve}{t+M},\\
J_4 &\geq C(t+M)^{-2}\t J_6,\\
J_5 &\leq \f{C_1}{(1+t)^\la}\,F'(t),
\end{align*}
and
\[
\t J_6 =\int_0^t(t-\tau)\int_{\tau+M_0}^{\tau+M}
\int_{|x|>y}\f{(|x|-y)^2}{|x|}\,\Phi_\g(\rho(\tau,x)-\bar\rho)\,dxdyd\tau.
\]
Denote $\O_1=\{(\tau,x)\colon \bar\rho\leq\rho(\tau,x)\leq2\bar\rho\},
\O_2=\{(\tau,x)\colon \rho(\tau,x)>2\bar\rho\}$, and
$\O_3=\{(\tau,x)\colon \rho(\tau,x)<\bar\rho\}$. Divide $F(t)$ into
a sum the three integrals over the domains $\O_i$ $(1\le i\le 3)$
\begin{equation*}
F(t)=F_1(t)+F_2(t)+F_3(t)\equiv\int_{\O_1}\cdots+\int_{\O_2}\cdots+\int_{\O_3}\cdots
\end{equation*}
Corresponding to the three parts of $F(t)$, we define $\t J_6\equiv\t
J_{6,1}+\t J_{6,2}+\t J_{6,3}$.  In view of $F(t)\geq0$ and
$F_3(t)\leq0$, we have
\begin{equation*}
F(t)\leq F_1(t)+F_2(t).
\end{equation*}
Applying H\"{o}lder's inequality for the domains $\O_1$ and $\O_2$,
we obtain that
\begin{align*}
F(t)&\leq \t
J_{6,1}^\f{1}{2}\left(\int_0^t(t-\tau)\int_{\tau+M_0}^{\tau+M}\f{1}{y^2}\int_{y<|x|\leq\tau+M}
\f{(|x|-y)^2}{|x|}\,dxdyd\tau\right)^\f{1}{2} \\ &\quad +\t
J_{6,2}^\f{1}{\g}
\left(\int_0^t(t-\tau)\int_{\tau+M_0}^{\tau+M}\f{1}{y^\f{\g}{\g-1}}\int_{y<|x|\leq\tau+M}
\f{(|x|-y)^2}{|x|}\,dxdyd\tau\right)^\f{\g-1}{\g} \\ &\leq \t
J_6^\f{1}{2}(t+M)^\f{1}{2}\log^\f{1}{2}(t/M+1)+\t
J_6^\f{1}{\g}(t+M)^\f{\g-1}{\g} \\ &=\bigl(\t
J_6(t+M)^{-1}\bigr)^\f{1}{2}(t+M)\log^\f{1}{2}(t/M+1)+\bigl(\t
J_6(t+M)^{-1}\bigr)^\f{1}{\g}(t+M).
\end{align*}
In view of $1<\g<2$, we have $\ds\f{1}{2\g}<\f{1}{2}<
\f{1}{\g}$. Applying Young's inequality yields
\begin{equation*}
F(t) \leq \left(\bigl(\t J_6(t+M)^{-1}\bigr)^\f{1}{2\g}+\bigl(\t
  J_6(t+M)^{-1}\bigr)^\f{1}{\g}\right)(t+M)\log^\f{1}{2}(t/M+1), \quad
t\geq \t t_1\equiv Me.
\end{equation*}
Together with the fact that $F(t)\geq C\ve(t+M)\log(t/M+1)$,
this yields
\begin{equation*}
\t J_6 \geq CF(t)^\g(t+M)^{1-\g}\log^{-\f{\g}{2}}(t/M+1), \quad t\geq
\t t_1.
\end{equation*}
Substituting this into \eqref{3.38} yields
\begin{align}
F''(t)+\f{C_1}{(1+t)^\la}F'(t) &\geq \f{C\ve}{t+M}, \quad t\geq0, \label{3.39}\\
F''(t)+\f{C_1}{(1+t)^\la}F'(t) &\geq CF(t)^\g(t+M)^{-1-\g}\log^{-\f{\g}{2}}(t/M+1), \quad t\geq \t t_1. \label{3.40}
\end{align}
Substituting $F(t)\geq C\ve(t+M)\log(t/M+1)$ into \eqref{3.40} yields
\begin{equation*}
F''(t)+\f{C_1}{(1+t)^\la}F'(t) \geq C\ve^\g(t+M)^{-1}\log^\f{\g}{2}(t/M+1).
\end{equation*}
Integrating this yields
\begin{equation*}
F(t) \geq C\ve^\g(t+M)\log^\f{\g+2}{2}(t/M+1).
\end{equation*}
Substituting this into \eqref{3.40} again gives
\begin{multline*}
F''(t)+\f{C_1}{(1+t)^\la}F'(t) \\ \geq
C\ve^{\g^2}(t+M)^{-1}\log^\f{\g(\g+1)}{2}(t/M+1)=C\ve^{\g^2}(t+M)^{-1}\log^\f{\g(\g^2-1)}{2(\g-1)}(t/M+1).
\end{multline*}
Repeating this process $n$ times, we see that
\begin{equation}\label{3.41}
F''(t)+\f{C_1}{(1+t)^\la}F'(t) \geq
C\ve^{\g^n}(t+M)^{-1}\log^\f{\g(\g^n-1)}{2(\g-1)}(t/M+1),
\end{equation}
where $n=\left[\log_\g2\right]$. Solving \eqref{3.41} yields
\begin{equation*}
F(t) \geq C\ve^{\g^n}(t+M)\log^{\f{\g(\g^n-1)}{2(\g-1)}+1}(t/M+1), \quad t\geq \t t_2,
\end{equation*}
where $\t t_2>0$ is a constant only depending on $\g$. Substituting
this into \eqref{3.40} derives
\begin{equation}\label{3.42}
  F''(t)+\f{C_1}{(1+t)^\la}F'(t) \geq CF(t)
  \ve^{\g^n(\g-1)}(t+M)^{-1}\log^\f{\g^{n+1}-2}{2}(t/M+1), \quad t\geq \t t_2,
\end{equation}
where $\f{\g^{n+1}-2}{2}>0$ by the choice of
$n=\left[\log_\g2\right]$. Since \eqref{3.42} is analogous to
\eqref{3.29},  as in Case~1, we can choose $\t
t_3=O\Bigl(e^{C\ve^{-\f{2\g^n(\g-1)}{\g^{n+1}-2}}}\Bigr)$
such that
\begin{equation*}\label{3.43}
F'(t) \geq C\ve^\f{\g^n(\g-1)}{2}(t+M)^{-1}\log^\f{\g^{n+1}-2}{4}(t/M+1)F(t), \quad t\geq \t t_3,
\end{equation*}
which is similar to \eqref{3.34} and yields
\begin{equation}\label{3.44}
F(t) \geq C\ve^{C_\g}(t+M)^\f{2(\g+2)}{\g-1}, \quad t\geq \t t_4\equiv C\t t_3,
\end{equation}
where $C_{\g}>0$ is a constant depending on $\g$.
Substituting \eqref{3.44} into \eqref{3.40} yields
\begin{equation}\label{3.45}
F''(t)+\f{C_1}{(1+t)^\la}F'(t) \geq C\ve^{C_\g} F(t)^\f{\g+1}{2}, \qquad t\geq \t t_4.
\end{equation}
Multiplying \eqref{3.45} by $F'(t)$ and integrating over the variable
$t$ as in Case 1, we have
\begin{equation*}
F'(t) \geq C\ve^{C_\g}F(t)^\f{\g+3}{4}, \quad t\geq \t t_5\equiv C\t t_4.
\end{equation*}
Together with $\g>1$ and the choice of $\t t_3$, this yields $T_{\ve}<\infty$.

\medskip

Both Case~1 and Case~2 complete the proof of Theorem~\ref{thm1.2}.
\end{proof}


\section{Blowup for large data}\label{chap4}

In this section, we establish a blowup result for large amplitude
smooth solutions of Eq.~\eqref{1.1} which is valid for all
$\la\ge0$. More precisely, instead of \eqref{1.1} we consider the
Cauchy problem
\begin{equation}\label{4.1}
\left\{ \enspace
\begin{aligned}
&\p_t\rho+\operatorname{div}(\rho u)=0,\\
&\p_t(\rho u)+\operatorname{div}(\rho u\otimes
  u+pI_3)=-\,\f{\mu}{(1+t)^{\la}}\,\rho u,\\
&\rho(0,x)=\bar\rho+\t\rho_0(x),\quad u(0,x)=\t u_0(x),
\end{aligned}
\right.
\end{equation}
where $\t\rho_0,\, \t u_0\in C_0^{\infty}(\R^3)$,
$\operatorname{supp}\t\rho_0,\,\operatorname{supp}\t\rho_0\subseteq
B(0, M)\equiv \{x\colon |x|\leq M\}$, and $\rho(0,\cdot)>0$.  Motivated by
the treatment of the special case of $\la=0$ in \cite{19}, we
introduce the functions
\begin{gather*}
H(t)\equiv\int_{\mathbb R^3}x\cdot(\rho u)(t,x)\,dx, \qquad
L(t)\equiv\int_{\mathbb R^3}\left(\rho(t,x)-\bar\rho\right)dx, \\
\al(t)\equiv(t+M)^2\left(L(0)+\f{4\pi^2\bar\rho}{3}\,(t+M)^3\right),
\end{gather*}
and also remind the reader of the definition of the function $\beta$
in \eqref{beta}.

Then we have the following result:

\begin{theorem}\label{thm4.1}
Suppose that
$L(0)\geq0$ and
\begin{equation}\label{4.2}
H(0)\,\int_0^{T^*}\f{d\tau}{\al(\tau)\b(\tau)}>1.
\end{equation}
for some $T^\ast>0$. Then $T<T^*$ holds for any solution
$(\rho,u)\in C^1([0,T]\times\mathbb R^3)$ of \eqref{4.1}.
\end{theorem}

\begin{proof}
From the first equation of \eqref{4.1}, we see that
\[
L'(t)=-\int_{\mathbb R^3}\operatorname{div}(\rho u)\,dx=0,
\]
which implies $L(t)=L(0)$.  Applying the second equation of
\eqref{4.1}, we find that
\[
H'(t) = \int_{\mathbb R^3} x\cdot\p_t(\rho u)(t,x)\,dx= \int_{\mathbb R^3}
x\cdot\left[-\operatorname{div}(\rho u\otimes u)-\na
  p-\f{\mu}{(1+t)^\la}\,\rho u\right]dx.
\]
An integration by parts gives
\begin{equation}\label{4.3}
  H'(t)+\f{\mu}{(1+t)^\la}\,H(t)=\int_{\mathbb
    R^3}\left(\rho|u|^2+3\bigl(p(\rho)-p(\bar\rho)\bigr)\right)dx.
\end{equation}
Note that the convexity of $p=A\rho^\g$ for $\g>1$ and $c(\bar\rho)=1$ imply that
\begin{equation}\label{4.4}
\int_{\mathbb R^3}\bigl(p(\rho)-p(\bar\rho)\bigr)\,dx\geq \int_{\mathbb R^3}
A\g\bar\rho^{\g-1}(\rho-\bar\rho)\,dx=L(0).
\end{equation}
Furthermore, by applying the Cauchy-Schwartz inequality to $H(t)$ and
taking into account $\operatorname{supp}u(t,\cdot)\subseteq B(0,
M+t)$ for any fixed $t\ge 0$, we have
\begin{multline}\label{4.5}
H(t)^2 \leq \left(\int_{\mathbb
  R^3}\rho|u|^2\,dx\right)\left(\int_{|x|\leq t+M}\rho|x|^2\,dx\right)
\\ \leq
(t+M)^2\left(L(0)+\f{4\pi^2\bar\rho}{3}\,(t+M)^3\right)\int_{\mathbb
  R^3}\rho|u|^2\,dx = \al(t)\int_{\mathbb R^3}\rho|u|^2\,dx.
\end{multline}
Substituting \eqref{4.4}-\eqref{4.5} into \eqref{4.3} yields
\begin{equation*}\label{4.6}
H'(t)+\f{\mu}{(1+t)^\la}\,H(t) \geq \f{H(t)^2}{\al(t)}+3L(0).
\end{equation*}
Together with $L(0)\geq0$ and $H(0)>0$ due to \eqref{4.2}, this shows
that $H(t)>0$ for all $t\in[0,T]$. Denoting $G(t)\equiv\b(t)H(t)$,
from \eqref{beta} and \eqref{4.6} we then get that
\begin{equation}\label{4.7}
G'(t) \geq \f{G^2(t)}{\al(t)\b(t)}.
\end{equation}
Now suppose that $T\geq T^*$. Then integrating \eqref{4.7} from $0$ to
$T$ yields
\[
\f{1}{H(0)}-\f{1}{G(T)} \geq \int_0^T\f{d\tau}{\al(\tau)\b(\tau)} \geq
\int_0^{T^*}\f{d\tau}{\al(\tau)\b(\tau)}
\]
which is a contradiction in view of $G(T)>0$ and \eqref{4.2}.

Thus, Theorem~\ref{thm4.1} has been proved.
\end{proof}


\bigskip

\begin{acknowledgement}
Yin Huicheng wishes to express his gratitude to Professor Michael
Reissig, Technical University Bergakademie Freiberg, Germany, for his
interests in this problem and some fruitful discussions in the past.
\end{acknowledgement}



\end{document}